\documentclass[11pt]{amsart}
\usepackage{amssymb,amscd,  latexsym, graphicx, mathrsfs, enumerate}

\newcommand{\nc}{\newcommand}
\numberwithin{equation}{section}
\newtheorem{theorem}{Theorem}[section]
\newtheorem{prop}[theorem]{Proposition}
\newtheorem{importnota}[theorem]{Important Notation}
\newtheorem{prblm}[theorem]{Problem}
\newtheorem{notation}[theorem]{Notation}
\newtheorem{caution}[theorem]{Caution}
\newtheorem{remark}[theorem]{Remark}
\newtheorem{lemma}[theorem]{Lemma}
\newtheorem{construction}[theorem]{Construction}
\newtheorem{corollary}[theorem]{Corollary}
\newtheorem{example}[theorem]{Example}
\newtheorem{conclusion}[theorem]{Conclusion}
\newtheorem{triviality}[theorem]{Triviality}
\newtheorem{proto}[theorem]{Prototype Quasifibration}
\newtheorem{cauex}[theorem]{Cautionary Example}
\newtheorem{propositiondef}[theorem]{Proposition-Definition}
\newtheorem{subth}{Nuisance}[theorem]
\newtheorem{ssubth}{ }[subth]
\newtheorem{conjecture}[theorem]{Conjecture}
\newtheorem{sidest}[theorem]{Side Story}
\newtheorem{miniexample}[theorem]{Example}
\theoremstyle{definition}
\newtheorem{defin}[theorem]{Definition}

\nc\tri[1]{\begin{triviality}}
\nc\side[1]{\begin{sidest}}
\nc\conj[1]{\begin{conjecture}}
\nc\prodef[1]{\begin{propositiondef}}
\nc\prt[1]{\begin{proto}}
\nc\lem[1]{\begin{lemma}}
\nc\sblm[1]{\begin{sublemma}}
\nc\pro[1]{\begin{prop}}
\nc\thm[1]{\begin{theorem}}
\nc\cor[1]{\begin{corollary}}
\nc\dfn[1]{\begin{defin}}
\nc\sthm[1]{\begin{subth}}
\nc\exm[1]{\begin{example}}
\nc\miniexm[1]{\begin{miniexample}}
\nc\plm[1]{\begin{prblm}}
\nc\rmk[1]{\begin{remark}}
\nc\subrmk[1]{\begin{subremark}}
\nc\ntn[1]{\begin{notation}}
\nc\cau[1]{\begin{caution}}
\nc\imn[1]{\begin{importnota}}
\nc\cax[1]{\begin{cauex}}
\nc\con[1]{\begin{construction}}
\nc\ssthm[1]{\begin{ssubth}}
\nc\cnc[1]{\begin{conclusion}}
\nc\elem{\end{lemma}}
\nc\esblm{\end{sublemma}}
\nc\eside{\end{sidest}}
\nc\econj{\end{conjecture}}
\nc\eprodef{\end{propositiondef}}
\nc\eprt{\end{proto}}
\nc\ethm{\end{theorem}}
\nc\ecor{\end{corollary}}
\nc\edfn{\end{defin}}
\nc\esthm{\end{subth}}
\nc\epro{\end{prop}}
\nc\etri{\end{triviality}}
\nc\eexm{\end{example}}
\nc\eminiexm{\end{miniexample}}
\nc\ermk{\end{remark}}
\nc\subermk{\end{subremark}}
\nc\eplm{\end{prblm}}
\nc\ecau{\end{caution}}
\nc\ecax{\end{cauex}}
\nc\eimn{\end{importnota}}
\nc\entn{\end{notation}}
\nc\econ{\end{construction}}
\nc\ecnc{\end{conclusion}}
\nc\essthm{\end{ssubth}}

\newcommand{\C}{\mathbb{C}}
\newcommand{\R}{\mathbb{R}}
\newcommand{\Q}{\mathbb{Q}}

\newcommand{\Z}{\mathbb{Z}}

\newcommand{\X}{\mathfrak{X}}

\newcommand{\A}{\mathbb{A}}

\newcommand{\p}{\mathfrak{p}}

\newcommand{\diag}{{\rm diag}}
\newcommand{\G}{\Gamma}

\newcommand{\ds}{\displaystyle}

\newcommand{\lra}{\longrightarrow}

\newcommand{\sgn}{{\rm sgn}}

\renewcommand{\Bbb}{\mathbb}
\newcommand{\f}{\textbf{f}}
\newcommand{\bG}{\bold G}
\newcommand{\bs}{\backslash}

\newcommand{\jt}{\frak{J}_2}

\title[Higher level cusp forms on the exceptional group of type $E_{7}$]
{Higher level cusp forms on exceptional group of type $E_{7}$}
\author{Henry H. Kim and Takuya Yamauchi}
\thanks{The first author is partially supported by NSERC. The second author is partially supported by 
 JSPS KAKENHI Grant Number (C) No.15K04787.}
\subjclass[2010]{}
\address{Henry H. Kim \\
Department of mathematics \\
 University of Toronto \\
Toronto, Ontario M5S 2E4, CANADA \\
and Korea Institute for Advanced Study, Seoul, KOREA}
\email{henrykim@math.toronto.edu}

\address{Takuya Yamauchi \\
Mathematical Inst. Tohoku Univ.\\
 6-3,Aoba, Aramaki, Aoba-Ku, Sendai 980-8578, JAPAN}
\email{yamauchi@math.tohoku.ac.jp}

\begin{document}
\begin{abstract} 
By using new techniques with the degenerate Whittaker functions found by Ikeda-Yamana, we construct higher level cusp form on $E_{7,3}^{\rm ad}$ (adjoint exceptional group of type $E_7$), called Ikeda type lift, from any Hecke cusp form whose corresponding 
automorphic representation has no supercupidal local components. This generalizes the results in \cite{KY} on level one forms.
But there are new phenomena in higher levels; first, we can handle any non-trivial central characters. Second,
the lift does not depends on the choice of an irreducible cuspidal constituent of the restriction of the Hecke cusp form to 
$SL_2$. Hence any twist of the cusp form gives rise to the same lift. We also compute the degree 133 adjoint $L$-function of the Ikeda type lift.
\end{abstract}
\maketitle
\tableofcontents

\section{Introduction}
Let $\bold{G}$ be  the adjoint exceptional algebraic group of type $E_{7,3}^{\rm ad}$ over $\Q$ and $\frak T\subset \C^{27}$ 
the corresponding bounded symmetric domain.  
In \cite{KY} we constructed holomorphic cusp forms of level one on $\frak T$ from cusp forms of level one for $SL_2$ over $\Q$ by using Ikeda's idea regarding 
a uniform structure of Siegel Eisenstein series and the theory of Jacobi forms of matrix index:   
In \cite{Ik1}, Ikeda originally gave a (functorial) construction of a Siegel cusp form for $Sp_{2n}$ (rank $2n$) from a normalized Hecke eigenform on the upper half-plane $\mathbb{H}$ with respect to $SL_2(\Z)$. 
After this work, his construction was generalized to unitary groups $U(n,n)$ \cite{Ik2}, $O^*(4n)$ \cite{Yamana}, 
quaternion unitary groups $Sp(n,n)$ \cite{IY}, symplectic groups $Sp_{2n}$ over any totally real field \cite{Ik4, Ik&H}, and $E_{7,3}^{\rm ad}$ by the authors as mentioned above.  

Recently, in \cite{IY, yamana1}, Ikeda and Yamana found a method, though the original idea came from Ikeda's paper \cite{Ik4}, of using the degenerate Whittaker functions and local Fourier-Jacobi expansions which 
make the proof not only simpler but also applicable to elliptic cusp forms for higher levels. 
We apply their method to $E_{7,3}$ and generalize 
the results in \cite{KY} to elliptic cusp forms for higher levels. Due to a special feature of $E_{7,3}$, our construction works
even for cusp forms with any non-trivial central characters.    

We now explain our main theorem. We refer to \cite{KY} for the several notations which appear below. Let $\psi_p:\Q_p\lra \C^\times$ be a standard non-trivial additive character for each prime $p$ and put $\psi_\f=\otimes'_{p<\infty}\psi_p$. 

Let $f=\sum_{n=1}^\infty a_f(n)q^n$ be an elliptic newform of weight $k$ with the central character $\chi$ so that 
the corresponding automorphic representation $\pi'=\otimes'_p \pi_p'=\pi_{\bf{f}}'\otimes \pi_{\infty}'$ of $GL_2(\A_\Q)$ has no supercupidal local components. Then $f$ is a non-CM form.
For each prime $p$, $\pi_p'=\pi(\mu_{1p},\mu_{2p})$, $\mu_{1p}\mu_{2p}=\chi_p$, or ${\rm St}_{GL_2}\otimes\omega_p$, 
$\omega_p^2=\chi_p$. Here
$\mu_{ip}: \Bbb Q_p^\times\rightarrow\Bbb C$ are (unitary) characters, which are both unramified for almost all $p$, and ${\rm St}_{GL_2}$ is the Steinberg representation which is the subrepresentation of 
$\pi(|\cdot|^{\frac 12}, |\cdot|^{-\frac 12})$. Let $\mu_p:=\mu_{1p}\mu_{2p}^{-1}$.
Let $S_{\bf{f}}$ be the set of finite places so that $\pi_p'={\rm St}_{GL_2}\otimes\omega_p$.

Let $\pi$ be any irreducible cuspidal constituent of the restriction of $\pi'$ to $SL_2(\Bbb A_\Q)$, and let $\pi=\pi_{\bf{f}}\otimes\pi_\infty$. Then for $p\notin S_f\cup \{\infty\}$, $\pi_p$ is any irreducible direct summand of $I_1(0,\mu_p)={\rm Ind}^{SL_2(\Q_p)}_{B(\Q_p)}\mu_p$. (Here if $\mu_p$ is a quadratic character, $I_1(0,\mu_p)$ is a sum of two irreducible representations.
Since $f$ is non-CM, by \cite{LL}, any choice of irreducible direct summands gives rise to a cuspidal representation of $SL_2$.
It turns out that our lift does not depend on the choice. Hence for any gr\"ossencharacter $\chi$, $\pi'\otimes\chi$ will give rise to the same lift.)

Here $B=TN_1$ is the Borel subgroup in $SL_2$. 
For $p\in S_{\bf{f}}$, $\pi_p=A_1(|\cdot|_p)$ is the unique submodule of $I_1(|\cdot|_p)={\rm Ind}^{SL_2(\Q_p)}_{B(\Q_p)} |\cdot|_p$. 
Let $n\in \Z_{>0}$.
For $p\notin S_{\bf{f}}\cup \{\infty\}$ and $\phi_p\in \pi_p$, define
$$w^{\mu_p}_n(\phi_p)=|n|^{\frac{1}{2}}_p L(1,\mu_p)\int_{N_1(\Q_p)}\phi_p\Big(
\left(\begin{array}{cc}
0 & 1 \\
-1 & 0 
\end{array}
\right)
\left(\begin{array}{cc}
1 & z \\
0 & 1 
\end{array}
\right) \Big)\overline{\psi(nz)}dz,
$$
where $L(s,\mu_p):=(1-\mu_p(p)p^{-s})^{-1}$ if $\mu_p$ is unramified. If $\mu_p$ is ramified, we let $L(s,\mu_p)=1$. If $p\in S_{\bf{f}}$, 
we can define $w^{|\cdot|_p}_n(\phi_p)$ as the restriction of the same integral to $A_1(|\cdot|_p)$. 
(replace $L(1,\mu_p)$ by $\zeta_p(2)=(1-p^{-2})^{-1}$.) 
Now we define a constant $c_n(f)$, which should not be confused with $a_f(n)$, by 
\begin{equation}\label{cn}
\int_{N_1(\A_\f)} \phi\Big(
\left(\begin{array}{cc}
1 & z \\
0 & 1 
\end{array}
\right)\Big)\overline{\psi_{\f}(nz)}dz=n^k\cdot c_n(f) \prod_{p|n\atop p\notin S_{\bf{f}}} w^{\mu_p}_n(\phi_p) \prod_{p|n\atop p\in S_{\bf{f}}} w^{|\cdot|_p}_n(\phi_p),
\end{equation}
for all distinguished vector $\phi=\otimes_p \phi_p\in \pi_{\bf{f}}$. 
Given an elliptic newform $f$, the constant $c_n(f)$ always exists and is uniquely 
determined by $f$ and its realization as an automorphic form on $SL_2(\A_\Q)$.

We define similar functionals for $\bold G$. Let $\bold P=\bold M\bold N$ be the Siegel parabolic subgroup such that  
$\bold M\simeq GE_{6,2}$. We naturally identify $\bold N$ with the exceptional Jordan algebra $\frak J$. We write $p_B\in \bold N,\ B\in \frak J$ under 
this natural identification.    
Let us consider the degenerate principal series representation
$$I(s,\mu_p):={\rm Ind}^{\bold G(\Q_p)}_{\bold P(\Q_p)} \, (\mu_p\circ \nu) |\nu(\cdot)|^s
$$
associated to $\mu_p$, where $\nu:P\lra GL_1$ is the similitude character. Then since $\bold G$ is an adjoint group, by \cite{W},
$I(0,\mu_p)$ is irreducible. (If $\bold G$ is a simply connected $E_7$ and $\mu_p$ is a quadratic character, $I(0,\mu_p)$ is reducible. It is a sum of two irreducible representations.) 
If $\mu_p=1$,
$I(1,1)$ is reducible, and has a unique submodule 
$A(|\cdot|_p)$. 
We will prove that for any $B\in \frak J(\Q)_+$, one can associate a suitable vector $w^{\mu_p}_B$ in 
the one-dimensional space ${\rm Hom}_{N(\Q_p)}(I(0,\mu_p),\psi^B_p)$, or a suitable vector $w^{|\cdot|_p}_B$ in ${\rm Hom}_{N(\Q_p)}(A(|\cdot|_p),\psi^B_p)$, where $\psi^B_p(Z)=\psi_p((B,Z))$ for $Z\in \frak J(\Q_p)$. 
For $\phi_p\in I(0,\mu_p)$, it is explicitly given by 
$$w^{\mu_p}_B(\phi_p)=|\det(B)|^{\frac{9}{2}}_p\prod_{i=0}^2 L(9-4i,\mu_p)\int_{N(\Q_p)}\phi_p(\iota\cdot p_Z) \overline{\psi^B(Z)}dZ.
$$
Consider the admissible representation 
$$A_{\bf{f}}(\pi_{\bf{f}})=\bigotimes_{p\notin S_{\bf{f}}\cup \{\infty\}} I(0,\mu_p)\otimes\bigotimes_{p\in S_{\bf{f}}} A(|\cdot|_p),
$$
and $\Pi=\Pi_\infty\otimes A_{\bf{f}}(\pi_{\bf{f}})$, 
where $\Pi_\infty$ is the holomorphic discrete series of the lowest weight $k+8$. The choice of $\Pi_\infty$ is predicted by Lemma \ref{awc}. 
Then we have the main theorem:

\begin{theorem}\label{main-thm1} Let $k\geq 2$, and $S_k(\frak T)$ be the space of cusp forms on $\frak T$ of weight $k$. Let $\textbf{e}(x)=e^{2\pi x\sqrt{-1}}$.
Then $\Pi$ is a cuspidal automorphic representation of $E_7^{\rm ad}$. More explicitly, for $\phi\in A_{\bf{f}}(\pi_{\bf{f}})$,
the holomorphic function
\begin{equation}\label{classical}
I_k(\phi)(Z)=\sum_{B\in J_3(\Z)^+}c_{\det(B)}(f)\det(B)^{\frac{k}{2}+4} 
\Big(\prod_{p|\det(B)\atop p\notin S_{\bf{f}}} w^{\mu_p}_B(\phi_p) \prod_{p|\det(B)\atop p\in S_{\bf{f}}} w^{|\cdot|_p}_B(\phi_p)\Big)\textbf{e}((B,Z)),
\end{equation}
is a Hecke eigen cusp form in $S_{k+8}(\frak T)$, and gives rise to $\Pi$. We note that $I_k(\phi)(Z)$ is exactly (\ref{rel-imp1}) in classical language.
\end{theorem}

Here an important point is that we can write down the Ikeda type lift quite explicitly. 
We use adelic languages to prove Theorem \ref{main-thm1}. 

We denote by $S_k(\Bbb H)$ the space of cusp forms on the upper half plane $\Bbb H$ of weight $k$. 
We also denote by $S^{{\rm new}}_{k}(\mathbb{H})^{{\rm ns}}$ the space generated by all elliptic newforms
so that the corresponding automorphic representation 
$\pi'=\pi_\infty'\otimes\otimes_p' \pi_p'$ does not have supercuspidal local components.
Note that $A_{\bf{f}}(\pi_{\bf{f}})$ does not depend on the choice of an irreducible cuspidal constituent of the restriction of the cusp forms in $S^{{\rm new}}_{k}(\mathbb{H})^{{\rm ns}}$
to $SL_2(\Bbb A_\Q)$. 
Hence we have the following map, which is not injective; if $f\in S^{{\rm new}}_{k}(\mathbb{H})^{{\rm ns}}$, then
for any gr\"ossencharacter $\chi$, $f\otimes\chi$ will give rise to the same lift.

\begin{corollary}\label{cor1} For any integer $k\ge 2$, there exists a $\C$-linear map  
$$S^{{\rm new}}_{k}(\mathbb{H})^{{\rm ns}}\lra S_{k+8}(\frak T)
$$
which is functorial in the sense that it preserves Hecke eigen forms. 
\end{corollary}

However, in some cases, we have an injective map:
For a positive integer $N$, define 
$\Gamma^{\bold{G}}_0(N)=\{g\in \Gamma\ |\ (g\ {\rm mod}\ N)\in {\bold P}(\Z/N\Z) \}$, where 
$\bold P$ is the Siegel parabolic subgroup in $\bold G$. 
For a positive integer $k\ge 10$, let $S_k(\Gamma^{\bold G}_0(N))$ be the space of 
cusp forms 
weight $k$ with respect to $\Gamma^{\bold G}_0(N)$ on $\frak T$. 

As usual, let $\Gamma_0(N)=\Bigg\{
\left(
\begin{array}{cc}
a & b \\
c & d
\end{array}
\right)\in SL_2(\Z)\ \Bigg|\ c\equiv 0\  {\rm mod}\ N 
\Bigg\}$.
For $k\geq 2$, let $S_k(\Gamma_0(N))^{{\rm new}}$ be the space of elliptic newforms of weight $k$ with respect to $\Gamma_0(N)$ with the trivial central character. For $f\in S_k(\Gamma_0(N))^{{\rm new}}$,
let $\pi'$ be the cuspidal representation of $GL_2(\Bbb A_\Q)$ associated to $f$.
Suppose $N$ is square free.
Then it is well-known \cite{LW} that if $p|N$,
$\pi_p'$ is the Steinberg representation, or the unramified twist of the Steinberg representation. 
It is determined by the  Atkin-Lehner eigenvalue: Let $W_p=\begin{pmatrix} 0&1\\p&0\end{pmatrix}$. Then for $f\in S_k(\Gamma_0(N))^{{\rm new}}$, $W_pf=\epsilon_p f$ with $\epsilon_p=\pm 1$. If $\epsilon_p=-1$, then $\pi_p={\rm St}_{GL_2}$. If $\epsilon_p=1$, then $\pi_p={\rm St}_{GL_2}\otimes\xi$, where $\xi$ is the unique non-trivial unramified quadratic character.
Let $S_k(\Gamma_0(N))^{{\rm new},-}$ be the space of new forms so that $\epsilon_p=-1$ for all $p|N$.

Suppose $\pi_1'$ and $\pi_2'$ are two cuspidal representations attached to two newforms in $S_k(\Gamma_0(N))^{{\rm new}}$, and 
irreducible cuspidal constituents of $\pi_1'|_{SL_2(\Bbb A_\Q)}$ and $\pi_2'|_{SL_2(\Bbb A_\Q)}$ give rise to the same cuspidal representation in $S_{k+8}(\frak T)$. Then $Ad(\pi_1')=Ad(\pi_2')$. By \cite[Theorem 4.1.2]{Ra}, $\pi_2'\simeq \pi_1'\otimes\omega$ for a gr\"ossencharacter $\omega$ with $\omega^2=1$. Since $N$ is square free, $\omega_p$ should be unramified for all $p$.
There is no such nontrivial gr\"ossencharacter. Hence $\omega=1$.
Therefore we have

\begin{theorem}\label{main-thm2} Let $N$ be square free. For any integer $k\ge 2$, there exists a $\C$-linear injective map  
$$S_{k}(\Gamma_0(N))^{{\rm new}} \lra S_{k+8}(\Gamma^{\bold G}_0(N))
$$
which is functorial. 
\end{theorem}

Let $\Pi$ be the cuspidal representation of $E_{7,3}^{\rm ad}$ associated to the cusp form on $S_{k+8}(\frak T)$. 
By \cite[p.158]{knapp}, $\Pi_\infty$ is a holomorphic discrete series if and only if $k\geq 10$. When $k=9$, it is the limit of holomorphic discrete series. 
It is known that dim $S_{10}(\Gamma_0(2))^{\rm new}=1$, and it is generated by $f=q+16q^2-156q^3+\cdots.$ So $f$ gives rise to a weight 18 cusp form on $\frak T$ with a holomorphic discrete series. It is known that dim $S_9(\Gamma_0(3), \chi)^{\rm new}=2$, where $\chi$ is the odd Dirichlet character mod 3. The two generators will both give rise to the same cusp form of weight 17 on $\frak T$ with the limit of holomorphic discrete series. We have dim $S_2(\Gamma_0(11))^{\rm new}=1$. It is generated by $f=q-2q^2-q^3+\cdots.$ It gives rise to a weight 10 cusp form on $\frak T$ where the infinity component is not a holomorphic discrete series, but some lowest weight representation.

Note that ${}^L E_7^{\rm sc}=E_7^{\rm ad}(\Bbb C)$, and $E_7^{\rm ad}(\Bbb C)$ does not admit the degree 56 representation. That is the reason why we should consider the Ikeda type lift $\Pi$ as a cuspidal representation of $E_{7,3}^{\rm ad}$.
In order to define the degree 56 standard $L$-function, we consider $\Pi$ as a cuspidal representation of $GE_{7,3}$ with the trivial central character.

\begin{theorem}\label{main-thm4} Assume that $f\in S_k(\Gamma_0(N))^{{\rm new},-}$.
Then the degree $56$ standard L-function $L(s,\Pi,{\rm St})$ of $\Pi$ is given by
$$L(s,\Pi,{\rm St})=L(s,{\rm Sym}^3 \pi_f)L(s,\pi_f)^2 \prod_{i=1}^4 L(s\pm i,\pi_f)^2 \prod_{i=5}^8 L(s\pm i, \pi_f),
$$
where $L(s,{\rm Sym}^3 \pi_f)$ is the symmetric cube $L$-function. 
\end{theorem}

\begin{remark}\label{L-fun} Note that for $f\notin S_k(\Gamma_0(N))^{{\rm new},-}$, we can define $L(s,\Pi,{\rm St})$, but it will not match with the right hand side.
The situation is similar to that of $PGSp_4$ and $GSp_4$. If we have a cuspidal representation $\pi$ of 
$PGSp_4$, one must extend $\pi$ to a representation $\pi'$ of $GSp_4$ with the trivial central character in order to define the degree 4 $L$-function. If $\pi_p'$ is unramified, the Satake parameter of $\pi_p'$ is $\{\alpha_{0p}, \alpha_{0p}\alpha_{1p}, \alpha_{0p}\alpha_{2p},\alpha_{0p}\alpha_{1p}\alpha_{2p}\}$, where $\alpha_{0p}^2\alpha_{1p}\alpha_{2p}=1$.
Hence there is ambiguity due to the sign of $\alpha_{0p}$.
\end{remark}

Finally, we compute the degree 133 adjoint $L$-function of $\Pi$. Even though $E_7^{\rm ad}(\Bbb C)$ does not admit the degree 56 representation, but it does admit the degree 133 representation which is the adjoint representation. 

\begin{theorem}\label{main-thm5} Assume that $f\in S_k(\Gamma_0(N))^{{\rm new}}$.
Then the degree 133 adjoint $L$-function $L(s,\Pi,{\rm Ad})$ of $\Pi$ is given by
\begin{eqnarray*}
&& L(s,\Pi,{\rm Ad})=\zeta(s)^4 \zeta(s\pm 1)^4 \zeta(s\pm 2)^2 \zeta(s\pm 3)^4 \zeta(s\pm 4)^3 \zeta(s\pm 5)^3
\zeta(s\pm 6)^2\zeta(s\pm 7)^2 \\
&& \phantom{xxxxxxxxx} \cdot \zeta(s\pm 8)\zeta(s\pm 9)\zeta(s\pm 10)\zeta(s\pm 11) \\
&& \phantom{xxxxxxxxx} \cdot L(s,\pi_f,{\rm Sym}^2)^3 \prod_{i=1}^4 L(s\pm i,\pi_f,{\rm Sym}^2)^2 \prod_{i=5}^8 L(s\pm i,\pi_f,{\rm Sym}^2),
\end{eqnarray*}
where $L(s,\pi_f,{\rm Sym}^2)$ is the symmetric square $L$-function. 
\end{theorem}
Unlike the degree 56 standard $L$-function, $L(s,\Pi,{\rm Ad})$ matches with the right hand side for any $f\in S_k(\Gamma_0(N))^{{\rm new}}$.

This paper is organized as follows. Since we are going to use the same notations in \cite{KY}, we will not recall the notations and basic facts from \cite{KY}. Instead we make some corrections in \cite{KY} in Section 2, since we need some of those facts. In Section 3, we supplement some facts we need in this paper. In Section \ref{modular}, we recall the definition of modular forms for congruence subgroups on the exceptional domain.
In Section 5, we study the degenerate Whittaker functions and Fourier-Jacobi coefficients in various settings. 
After reviewing well-known facts in Section 6 for automorphic forms on ${\rm SL}_2(\A_\Q)$, we study local analogue of Fourier-Jacobi expansion in Section 7.
The compatibility of the local and global computations play an important role to prove the main theorem. 
A basic idea is similar to the one initiated by Ikeda which goes as follows: first, we consider 
a formal series defined by Satake parameters of an elliptic modular form and the 
reciprocal polynomial defined by Siegel series; second, we prove its automorphy. 
One key observation due to Ikeda and Yamana is to lift ``automorphy" of an 
elliptic modular form (or automorphic forms in the cuspidal representation of the modular form in question) 
by checking the compatibility as mentioned.  
In Section 8, we will complete proofs of the main theorems.

In Section 9, we interpret the Ikeda type lift as a CAP representation (Cuspidal representation Associated to a Parabolic subgroup). The Ikeda lift on $Sp_{2n}$ is a CAP representation. However, the Ikeda type lift on $E_{7,3}^{\rm ad}$ is not a CAP representation in the usual sense since there are not many $\Bbb Q$-parabolic subgroups of $E_{7,3}$. We show that the Ikeda type lift on $E_{7,3}^{\rm ad}$ is a CAP representation in a more general sense.

Finally, we remark that W.T. Gan and H.Y. Loke \cite{GL} constructed modular forms on $E_{7,3}$ of level $p$. They are not cusp forms. They are obtained by computing the action of the Iwahori-Hecke algebra at $p$ on the adelic form of the modular form of weight 4 constructed in \cite{kim}.

\smallskip

\noindent\textbf{Acknowledgments.} We would like to thank T. Ikeda and S. Yamana for helpful discussions and corrections. 

\section{Corrections and Remarks to \cite{KY}}\label{Cayley}

We would like to make the following corrections to \cite{KY}. It does not affect the main result of the paper.

page 228, line 7: The action of $\bold M$ should be
$g(X,\xi,X',\xi')=(g^*X, \nu(g)\xi, gX', \nu(g)^{-1}\xi')$. It agrees with the Siegel modular group case, namely, $A\in GL_n$ is embedded in $Sp_{2n}$ as diag$(A, {}^t A^{-1})$, rather than diag$({}^t A^{-1}, A)$.

page 229, line 2: $B_1=\nu(g) g^*B$. The proof of Lemma 3.2 should be modified.

page 230, line -16: 
$$h(a) \begin{pmatrix} X\\ \xi\\X'\\ \xi'\end{pmatrix}=\begin{pmatrix} X+4(a^{-1}-1)e_3\times (e_3\times X)+(a-1)(e_3,X)e_3\\ a\xi\\ X'+(a^{-1}-1)(e_3,X')e_3+4(a-1)e_3\times (e_3\times X')\\ a^{-1}\xi'\end{pmatrix}.
$$

page 230, line -1: $\sigma_S(x,y)=\frac 12 (S,\sigma(x,y))$.

page 234, line -15: $kZ=\nu(k)(k^*X+k^*Y\sqrt{-1})$.

page 237, line -3: In the Fourier expansion of $E_{2l}(Z)$, $T\in \frak J(\Bbb Z)_{\geq 0}$.

page 238, line 1: $\Gamma(2l-4n)$.

page 238, line 18: $f(pg)=\delta_{\bold P}^{1/2}(p)|\nu(p)|_{\Bbb A}^s f(g)$.

page 240, line -10: $B=m_{xe_{12}}^* z=m_{-\bar x e_{21}} z=p_{z'}$.

page 242, line -1: $\iota\cdot v_1(x,y,z)=v_1(0,-\overline{x},0)\cdot \iota \cdot v_1(0,y,z+\sigma(x,y))$.

page 243, line 2: $\lambda_S$ should be replaced with the modified ``$\sigma_S$" as above. Accordingly we should have 
\begin{eqnarray}
R(h;f,\varphi)&=&\int_{Y(\A)}\int_{Z(\A)}f(\iota\cdot v_1(0,y,z)\iota_{e_3}h)
\overline{(\omega^\psi_S(h)\phi)(-y)\psi_S(z)}dydz \nonumber \\
&=&\int_{Y(\A)}\int_{Z(\A)}f(\iota\cdot\iota_{e_3}\cdot v_1(-y,0,z)h)
\overline{(\omega^\psi_S(h)\phi)(-y)\psi_S(z)}dydz \nonumber \\
&=&\int_{Y(\A)}\int_{Z(\A)}f(\iota\cdot\iota_{e_3}\cdot v_1(y,0,z)h)
\overline{(\omega^\psi_S(v_1(y,0,z)h)\phi)(0)}dydz.  \nonumber
\end{eqnarray}

page 245, line -10: It should be $e^{-2\pi (T,Y)}$ in the integrand.

page 249: In Section 11, we should take $\bold G$ to be the adjoint exceptional group of type $E_{7,3}^{\rm ad}$ so that $\pi_F$ is a cuspidal representation of $E_{7,3}^{\rm ad}$. In order to define the degree 56 standard $L$-function, we need to first extend $\pi_F$ to a cuspidal representation $\Pi$ of $GE_{7,3}$ with the trivial central character. 

\section{Supplements to \cite{KY}} 

We keep notations as in \cite{KY} and add some facts we need in this paper.
Recall that $\bold H$ is the group generated by $\bold U_{2e_3}$ and $\iota_{e_3}$, and $\bold H\simeq SL_2$. 
We let $h(a)\in \bold H$ correspond to $\begin{pmatrix} a&0\\0&a^{-1}\end{pmatrix}$; let $n(b)\in\Bbb H$ correspond to $\begin{pmatrix} 1&b\\0&1\end{pmatrix}$. Note that $\iota_{e_3}$ corresponds to $\begin{pmatrix} 0&1\\ -1&0\end{pmatrix}$.

\begin{lemma} \label{useful}
\begin{enumerate}
\item $v_1(x,0,0)v_1(0,y,z-\sigma(x,y))=v_1(x,y,z)$ 
\item Let $X\in \frak J_2$. Then $v_1(x,0,0)\in \bold M$ and it acts on 
$X\oplus r:=\left(\begin{array}{cc}
X & 0 \\
0 & r
\end{array}
\right)\in \frak J\simeq \bold N$ as 
$$R_{r,x}:=v_1(x,0,0)\cdot (X\oplus r)=\left(\begin{array}{cc}
X & Xx \\
{}^t\overline{x}X & r+\sigma_R(x,x)
\end{array}
\right)$$ 
via the natural identification $\frak J\simeq \bold N$,\, $B\mapsto p_B$. 
\item $v_1(0,y,z)\in \bold N$ and it is identified with 
$\left(\begin{array}{cc}
z & y \\
{}^t \overline{y} & 0
\end{array}
\right)\in \frak J$ under $\frak J\simeq \bold N$. 
\item For $m\in \bold M$ and $B,Z\in \bold N$, 
$$(mB,Z)=(B,mZ).
$$
\item $n(z')\cdot v_1(0,y,z)=p_B$ where $B=\left(\begin{array}{cc}
z & y \\
{}^t\overline{y} & z'
\end{array}
\right)$. 
\item For $h=h(\gamma)\in \bold H(\R),\ \gamma\in {\rm SL}_2(\R)$, 
$h(E\sqrt{-1})=\diag(\sqrt{-1},\sqrt{-1},\gamma(\sqrt{-1}))$, where $E=diag(1,1,1)$.
\end{enumerate}
\end{lemma}
\begin{proof}
The first and third claim are clear. The second claim follows from Lemma 3.2 of \cite{KY}. 

For (4), it is true for a generator of $M$ such as $m_{xe_{ij}}$ by (3.1) of \cite{KY}.
Hence it is true for all elements in $\bold M$.  

For (6), since $\bold H$ is generated by $n(b)$ and $\iota_{e_3}$, by the cocycle relation,
it is enough to prove that for $Z=\begin{pmatrix} Z'&w\\{}^t\bar w&\tau\end{pmatrix}$,
$$n(b)Z=\begin{pmatrix} *&*\\ *&\tau+b\end{pmatrix},\quad \iota_{e_3}Z=\begin{pmatrix} *&*\\ *&-\frac 1{\tau}\end{pmatrix}.
$$
The first equality is obvious since $n(b)p_Z=p_{Z+be_3}$. For the second equality, write
$\iota_{e_3}p_Z=p_{Z_1}kp_A'$. Then $Z_1=\iota_{e_3}Z$.
Since $\iota_{e_3}=p_{e_3}p_{-e_3}'p_{e_3}$, we have
$p_{-e_3}'p_{e_3}p_Z=p_{Z_1-e_3}kp_A'$. By explicit computation, we can see that $\nu(k)^{-1}=-\tau$,
and $Z_1-e_3=\begin{pmatrix} *&*\\ *&-1-\frac 1{\tau}\end{pmatrix}$.
\end{proof}

\section{Modular forms for congruence subgroups on the exceptional domain}\label{modular}

Let $\bold G$ be the adjoint exceptional group of type $E_{7,3}^{\rm ad}$. It is defined as $\bold G^{\rm sc}/\{\pm 1\}$, where $
\bold G^{\rm sc}$ is the exceptional group defined in \cite{B}.
Let $\bold T$ be a fixed maximal torus of $\bold G$. Let $\bold B$ be the standard Borel subgroup containing $\bold T$.
Let $\{\beta_1, \dots, \beta_7 \}$ be the set of simple roots of $\bold T$ in $\bold B$, numbered as in Bourbaki \cite{Bour}.
The Satake diagram of $E_{7,3}^{\rm ad}$ is
\begin{align*}
{\text{o}}_{\beta_1}\text{------}{\bullet}_{\beta_3}{%
\text{------}}&{\bullet}_{\beta_4}{\text{------}}{\bullet}_{\beta_5}{\text{------}}{\text{o}}_{\beta_6}{\text{------}}{%
\text{o}}_{\beta_7} \\
&\Big\vert \\
&{\bullet}_{\beta_2}
\end{align*}
The $\Bbb Q$-root system is of type $C_3$, and the Dynkin diagram of $C_3$ is

\begin{align*}
{\text{o}}_{\lambda_1}{\text{------}}{\text{o}}_{\lambda_2}{\Longleftarrow}{\text{o}}_{\lambda_3},
\end{align*}
where $\lambda_1$ corresponds to $\beta_1$, $\lambda_2$ to $\beta_6$, $\lambda_3$ to $\beta_7$. Here $\lambda_1,\lambda_2$ have multiplicity 8, and $\lambda_3$ has multiplicity 1.

Let $\bold P=\bold M\bold N$ be the Siegel parabolic subgroup associated to $\lambda_3$. Then $\bold M\simeq GE_{6,2}$.
Let $\bold Q=\bold L\bold V$ be the maximal parabolic subgroup associated to $\lambda_2$. Then the derived group of $\bold L$ is $\bold H\times Spin(1,9)$, and $\bold V$ is the Heisenberg group.

Let $\Gamma=\bold G(\Z)$ be the arithmetic subgroup of $\bold G(\Q)$.
We define the principal congruence subgroup of level $N (\in \Z_{>0})$ by 
$$\Gamma^{\bold G}(N):={\rm Ker}(\bold G(\Z)\stackrel{{\rm mod}\ N}{\lra} \bold G(\Bbb Z/N\Bbb Z)).
$$ 
We call a finite index subgroup $\Gamma$ in $\bold G(\Z)$ a congruence subgroup if $\Gamma$ contains 
$\Gamma^{\bold G}(N)$ for some $N\in \Z_{>0}$. In fact, any finite index subgroup is a congruence subgroup.

For a function $F:\frak T\lra \C$, $k\in \Z_{\ge 0}$, and $g\in \bold G(\R)$, we define the ``slash operator" by 
$$F|[g]_k(Z):=j(g,Z)^{-k}F(gZ),
$$
where $j(g,Z)$ is the canonical factor of automorphy in \cite{KY}.

\begin{defin}Let $\Gamma$ be a congruence subgroup in $\bold G(\Z)$. 
 Let $F$ be a holomorphic function on $\frak T$ which for some integer $k>0$ satisfies
$$F|[g]_k(Z)=F(Z),\quad Z\in \frak T,\, \gamma\in\Gamma.
$$
Then $F$ is called a modular form on $\frak T$ of weight $k$ with respect to $\Gamma$. We denote by $\mathcal{M}_k(\Gamma)$ the space of such forms. 
For a holomorphic function $F:\frak T\lra \C$, 
the boundary map $\Phi$ is defined by 
$$\Phi F(Z')=\lim_{\tau\to \sqrt{-1}\infty}
F\left(\begin{array}{cc}
Z' & \ast \\
{}^t\bar{\ast} & \tau
\end{array}
\right),
$$
where $Z'\in \frak T_2$.
We call $F$ a cusp form of weight $k$ with respect to $\G$ if $F\in \mathcal{M}_k(\Gamma)$ and 
$\Phi F|[g]_k(Z)=0$ for any $g\in \bold G(\Q)$. We denote by $\mathcal{S}_k(\Gamma)$ the space of 
cusp forms of weight $k$ with respect to $\G$.  
\end{defin}
By Koecher principle, the holomorphy at the cusps follows automatically. 
If $F$ is a cusp form, $a_{F|[g]_k}(T)=0$ for $T\notin \frak J(\Z)_{>0}$ and $g\in \bold G(\Q)$ where $a_{F|[g]_k}(T)$ stands for 
the $T$-th Fourier coefficient of $F|[g]_k$. 
The converse is also true. 
We define $$\mathcal{S}_k(\frak T)=\bigcup_{\G}\mathcal{S}_k(\G)$$ where the union runs over 
all congruence subgroups in $\bold G(\Z)$.  

In this paper, we work with adelic language. The translation between adelic language and classical language is easy by the strong approximation theorem as in \cite[(5.3)]{KY}, namely, if $F$ is a modular form in $\mathcal{M}_k(\Gamma)$, we can define an automorphic form $\tilde F$ on $G(\Bbb A)$ by 
$$\tilde F(g)=j(g_\infty, E\sqrt{-1})^{-k} F(g_\infty(E\sqrt{-1})) \quad \text{for $g=\gamma\cdot g_\infty\cdot k'\in \bold G(\Bbb Q)\bold G(\Bbb R)K$}.
$$

\section{Degenerate Whittaker functions}\label{degenerate}
As in Section 3 of \cite{IY}, we study the degenerate Whittaker functions for the exceptional group $\bold G=E_{7,3}^{\rm ad}$ in 
various settings. In this section we often use a natural identification $\frak J\simeq \bold N$ given by $B\mapsto p_B$.

\subsection{Degenerate principal series: non-archimedean case} 
Let $p$ be a rational prime. Let $K_p:=\bold G(\Z_p)$. Then $\bold G(\Bbb Q_p)$ is a split adjoint group of type $E_7$. 
Let $\psi:\Q_p\lra \C^\times$ be the standard non-trivial additive character (defined as the composition of these homomorphisms:
$\psi: \Bbb Q_p\rightarrow \Bbb Q_p/\Bbb Z_p\rightarrow \Bbb Q/\Bbb Z\rightarrow \Bbb R/\Bbb Z$).
For $B\in \frak J(\Q_p)$, we define 
$\psi_S:\bold N(\Q_p)\lra \C^\times$,\, $p_Z\mapsto \psi_S(Z)=\psi((S,Z))$. 
For $\mu:\Q^\times_p\lra \C^\times$ a (unitary) character and $s\in\Bbb C$,
let $I(s,\mu)$ be the degenerate principal series representation of $\bold{G}(\Q_p)$ consisting of any smooth $K_p$-finite     
function $\phi:\bold{G}(\Q_p)\lra \C$ such that 
$$\phi(pg)=\delta^{\frac{1}{2}}_{\bold P}(p)\, \mu(\nu(p)) |\nu(p)|^s \phi(g)
$$
for any $p\in \bold P(\Q_p)$ and $g\in \bold{G}(\Q_p)$,
where $\bold P=\bold M\bold N$ is the Siegel parabolic subgroup.
Note that the modulus character $\delta_P$ is given by $\delta_{\bold P}(mn)=|\nu(m)|^{18}_{p}$. We denote it also by 
$I(s,\mu)=\mbox{Ind}_{\bold P(\Q_p)}^{\bold G(\Q_p)} \ (\mu\circ \nu) |\nu(\cdot) |^s$.
For any smooth representation $\Pi$ of $\bold G(\Q_p)$ we put 
$${\rm Wh}_B(\Pi):={\rm Hom}_{N(\Q_p)}(\Pi,\psi_S).
$$
\begin{prop}\label{irr-uni} It holds that 
\begin{enumerate}
\item If $-1<Re(s)<1$, then $I(s,\mu)$ is irreducible and unitary. 
\item ${\rm dim}{\rm Wh}_B(I(s,\mu))=1$ for all character $\mu$ and $B\in R_3(\Q_p)$. 
\item If $\mu$ is non-trivial, $I(1,\mu)$ is irreducible and unitary; $I(1,1)$ has a unique irreducible submodule $A(|\cdot|_p)$ which is unitary. Moreover, ${\rm dim}\, {\rm Wh}_B(A(|\cdot|_p))=1$ for all $B\in R_3(\Q_p)$. 
\end{enumerate} 
\end{prop}
\begin{proof} The irreducibility of (1) and (2) follows from \cite{W} and \cite{Ka}. 
The unitarity follows from \cite{Tadic}. 
 Here it is necessary that we take $\bold G$ to be the adjoint group. Indeed if $\bold G$ is simply connected and $\mu$ is quadratic, then $I(0,\mu)$ is reducible and a sum of two irreducible representations. 

For the last claim by Lemma 5.1.1  and Corollary 5.1.2 of \cite{W} there exists a unique subquotient $V$ of $I(-1,1)$ such that 
$0\lra |\nu|^8_p\lra I(-1,1)\lra V\lra 0.$
By taking the contragredient representations, we have 
$$0\lra V^\vee \lra I(1,1)\lra |\nu|^{-8}_p \lra 0.
$$
Since the length of $I(1,1)$ is at most two by Corollary 5.1.2 of \cite{W} again, $V^\vee$ is uniquely determined up 
to isomorphism and we may denote it by $A(\mu)$. The latter part of (3) follows from the above exact sequence since 
$B$ is non-degenerate, hence ${\rm Wh}_B(|\nu|^{-8}_p)=0$.  
\end{proof}

\subsection{Jacquet integrals and Siegel series} 
Let us keep the notations in the previous subsection. 
For $u\in \C$, we define the function $\varepsilon_u$ on $\bold G(\Q_p)$ by 
$$\varepsilon_u(g)=|\nu(p)|^u,\ g=pm\in \bold G(\Q_p)=\bold P(\Q_p)K_p.
$$
Recall $R_3(K)=\{ X\in \frak J_K\ |\, det(X)\ne 0\}$.
For $B\in R_3(\Q_p)$ and $\phi\in I(s,\mu)$, we define the Jacquet integral by 
$$\textbf{w}^{\mu,s,u}_B(h):=\int_{N(\Q_p)}(\phi\cdot \varepsilon_u)(\iota p_z)\overline{\psi_B(z)}dz,\ u\in \C.
$$
It is absolutely convergent for ${\rm Re}(u)>9-Re(s)$ and in fact it is a polynomial in $p^{-u}$ (cf. \cite{Ka}). 
Therefore we can evaluate $\textbf{w}^{\mu,s,u}_B(h)$ at $u=0$ and put $\textbf{w}^{\mu,s}_B(h):=\textbf{w}^{\mu,s,0}_B(h)$. 
The functional $\textbf{w}^{\mu,s}_B$ defines a vector in ${\rm Wh}_B(I(s,\mu))$. 

In what follows we assume that $Re(s)>-1$ and consider 
\begin{equation}\label{nw}
w^{\mu,s}_B(\phi)=|\det(B)|^{\frac{9}{2}}_p\prod_{j=0}^2 L(9-4j+s,\mu)\textbf{w}^{\mu,s}_B(\phi)
\end{equation} 
where $L(s,\mu)=(1-\mu(p)p^{-s})^{-1}$ if $\mu$ is unramified. If $\mu$ is ramified, $L(s,\mu)=1$.  
Then we have 
\begin{lemma}\label{degp} Assume $B\in R_3(\Q_p)$ and $Re(s)>-1$. Then 
\begin{enumerate}
\item the functional $w^{\mu,s}_B\in {\rm Wh}(I(s,\mu))$ is non-zero and the restriction of 
$w^{\mu,1}_B$ to $A(|\cdot|_p)$ is also non-zero when $s=1$ and $\mu=1$;
\item for $m\in M(\Q_p)$ and $\phi\in I(s,\mu)$, $w^{\mu,s}_{B}(m\cdot \phi)=\mu(\nu(m))^{-1}|\nu(m)|^{-1}w^{\mu,s}_{m\cdot B}(\phi)$, 
\newline where $m\cdot \phi(g)=\phi(gm)$, and $m\cdot B$ is defined by $p_{m\cdot B}=m p_B m^{-1}$.    
\end{enumerate} 
\end{lemma}
\begin{proof}The first claim follows from Proposition \ref{irr-uni}. The second claim follows from the usual change of variables.  
\end{proof}

\begin{lemma}\label{fc1} Assume $Re(s)>-1$. Let $\phi\in I(s,\mu)$. For any compact subset $C$ of $\bold G(\Q_p)$ there is a 
non-negative valued Schwartz function $\Phi$ on $\frak J(\Q_p)$ and a constant $M>0$ such that 
$$|w^{\mu,s}_B(c\cdot \phi)|\le M|\Phi(B)|
$$
for all $c\in C$ and $B\in R_3(\Q_p)$. 
\end{lemma}
\begin{proof} The claim is similar to Lemma 3.3 of \cite{yamana2} and the proof there is applicable to this case.
\end{proof}

For simplicity, we denote $w^{\mu,0}_B$ by $w^\mu_B$, and the restriction of 
$w^{\mu,1}_B$ to $A(|\cdot|_p)$ by $w^{|\cdot|_p}_B$.

For $B\in R_3(\Q_p)$, we denote by $b(B,s)$, the Siegel series defined in \cite{Ka1}. 
Put $\gamma(s)=\ds\prod_{j=0}^2 \zeta_p(s-4j)$ where $\zeta_p(s)=(1-p^{-s})^{-1}$. Then it is well-known (cf. \cite{Ka1}) that  
there exists a polynomial $f_B(X)\in \Z[X]$ such that $f_B(p^{9-s}):=\gamma(s)b(B,s)$ (cf. \cite{Ka1} or Section 6 of \cite{KY}). 
Then by \cite{Ka1}, since $\varepsilon_s\in I(s-9,1)$, for $B\in R_3(\Q_p)$,
$$\textbf{w}_B^{1,s-9}(\varepsilon_s)=b(B,s)=\gamma(s)^{-1}f_B(p^{9-s}).
$$

\begin{lemma}\label{fc2} We have
$|w^{1,s-9}_B(\varepsilon_s)|\le |\det(B)|^{-\frac {15}2}_p$ for ${\rm Re}(s)>9$. 
\end{lemma}
\begin{proof} 
In the proof of Lemma 9.3 of \cite{KY}, we have
$f_B(X)=X^d+a_1X^{d-1}+\cdots +a_d$, where $p^{-d}=|det(B)|_p$ and $|a_i|\leq p^{11d}$ for all $i$.
Hence if $Re(s)>9$, $|f_B(p^{9-s})|\leq p^{12d}$.
So if ${\rm Re}(s)>9$, our estimate holds.
\end{proof}

\subsection{Degenerate Whittaker functions: the archimedean case}\label{dwa} 
Recall that $\textbf{e}(x)=e_\infty(x)=e^{2\pi \sqrt{-1}x}$ for $x\in \R$ which gives an additive character of $\R$. 
For $B\in \frak J(\R)_{>0}$ and $\ell\in \Z$, we define 
$$W^{(\ell)}_B(g):=\det(B)^{\frac{\ell}{2}}\textbf{e}((B,g(E\sqrt{-1})))j(g,E\sqrt{-1})^{-\ell},\ g\in \bold G(\R).
$$
By using the properties of the factor of automorphy $j(g,Z)$ and Lemma \ref{useful}-(4), one can check that 
\begin{equation}\label{whi}
W^{(\ell)}_B(p_{Z'}m g k)=\textbf{e}((B,Z'))\sgn(\nu(m))^\ell W^{(\ell)}_{m\cdot B}(g)j(k,E\sqrt{-1})^{-\ell} 
\end{equation}
for $Z'\in \frak J(\R),\ m\in \bold M(\R), g\in \bold G(\R)$, and $k\in K_\infty $ where $K_\infty={\rm Stab}_{\bold G(\R)}(E\sqrt{-1})$ is the maximal compact subgroup of $\bold G(\R)$. Note that $K_\infty$ is the compact form of the exceptional group $E_6$.  
Note that $m\cdot B$ stands for $mp_Bm^{-1}=p_{m\cdot B}$ and 
the readers should not confuse with $mB$ which means $m\in {\rm Aut}(\frak J)$ as in the definition of $\bold M\simeq GE_{6,2}$. 

Similarly for ${\rm SL}_2(\R)$, we consider the degenerate Whittaker functions as follows. 
For $\ell\in \Z$ and $r\in \R_{>0}$, define 
$$W^{(\ell)}_r(h):=r^{\frac{\ell}{2}}\textbf{e}(r(h\sqrt{-1}))j_{{\rm SL}_2}(h,\sqrt{-1})^{-\ell},\ h\in {\rm SL}_2(\R)$$
where $j_{{\rm SL}_2}(h,\tau):=c\tau +d$ is the canonical automorphic factor for 
$h=
\left(\begin{array}{cc}
a& b\\
c& d 
\end{array}
\right)\in 
{\rm SL}_2(\R)$ and $\tau\in \mathbb{H}$. 
It is easy to see that 
\begin{equation}\label{sl2whi}
W^{(\ell)}_r\left(\left(\begin{array}{cc}
1& z'\\
0& 1 
\end{array}
\right)\left(\begin{array}{cc}
a& 0\\
0& a^{-1} 
\end{array}
\right)\cdot h\cdot k \right)=\textbf{e}(rz')\sgn(a)^\ell W^{(\ell)}_{a^2 r}(h)j_{{\rm SL}_2}(k,\sqrt{-1})^{-\ell}
\end{equation}
for $z'\in \R,\ a\in \R^\times, h\in {\rm SL}_2(\R)$, and $k\in SO_2(\R)$. 
Under the natural identification $\bold H(\R)\simeq {\rm SL}_2(\R)$, note that if we put $\tau:=h\sqrt{-1}$, 
\begin{equation}\label{autop}
j(h,E\sqrt{-1})=j_{{\rm SL}_2}(h,\sqrt{-1}).
\end{equation} 

\subsection{Degenerate Whittaker functions: the global case} 
For each finite prime $p$, suppose that 
we have a degenerate principal series $I(s_p,\mu_p)$ with the following conditions:  
\begin{enumerate}
\item $s_p=0$ or $s_p=1$; 
\item $s_p=0$ for almost all $p$;
\item $\mu_p=1$ whenever $s_p=1$.
\end{enumerate} 

Put $\mu_\f:=\otimes'_{p<\infty} \mu_p$ which gives rise to a character of the finite ideles $\A^\times_{\f}$. 
Let $S_{\mu_\f}$ be the set of all primes $p$ such that $s_p=1$. Note that $S(\mu_\f)$ is a finite set by the definition.  
Define the degenerate principal series $\otimes_p I(s_p,\mu_p)$ of $\bold G(\A_{\textbf{f}})$ in a similar way as in the $p$-adic case. 
For $B\in R_3(\Q)$ and $\phi=\otimes'_{p} \phi_p\in \bigotimes_{p\not\in S(\mu_\f)\cup\{\infty\}} I(0,\mu_p) \otimes \bigotimes_{p\in S(\mu_\f)} A(|\cdot|_p)$, put  
$$w^{\mu_\f}_B(\phi)=\prod_{p\notin S(\mu_\f)\cup \{\infty\}} w^{\mu_p}_B(\phi_p)\times \prod_{p\in S(\mu_\f)} w^{|\cdot|_p}_B(\phi_p).
$$
Let us define 
$$A(\mu_\f):=\bigotimes_{p\not\in S(\mu_\f)\cup\{\infty\}} I(0,\mu_p)\otimes\bigotimes_{p\in S(\mu_\f)} A(|\cdot|_p).
$$
Then $A(\mu_\f)$ is a subrepresentation of $\otimes_p I(s_p,\mu_p)$ and it turns out to be an admissible generic representation of 
$\bG(\A_\f)$ by the following proposition.  

\begin{prop}\label{whi-glo} For any $B\in R_3(\Q)$, the restriction of the global 
Whittaker functional $w^{\mu_\f}_B(\phi)$ to $A(\mu_\f)$ is non-zero and ${\rm dim}\ {\rm Wh}_B(A(\mu_\f))=1$. 
\end{prop}
\begin{proof} It follows from Proposition \ref{irr-uni} and Lemma \ref{degp}. 
\end{proof}

\subsection{Automorphic forms on $\bG(\A_\Q)$}

Let $\psi$ be a non-trivial additive character of $\Q\bs\A_\Q$ 
so that $\psi_\infty=\textbf{e}_\infty$ and put $\psi_B(z)=\psi((B,z))$ for $B,z\in \frak J(\A_\Q)$. 
We factor $\psi_B$ as $\psi=\psi_{B,\f}\psi_{B,\infty}$. 
Let $\ell$ be a non-negative integer. 
For a smooth function $F: \bold N(\Q)\bs \bG(\A_\Q)\lra \C$ and $B\in \frak J(\Q)\simeq \bold N(\Q)$ we define by  
$$W_B(g,F):=\int_{\bold N(\Q)\bs \bold N(\A_\Q)}F(p_z\cdot g)\overline{\psi_B(z)}dz, g\in \bG(\A_\Q)
$$
the $B$-th Fourier coefficient of $F$. We say that $F$ admits ``the holomorphic Fourier expansion of the weight $\ell$" if $F$ takes the following form 
$$F(g)=\sum_{B\in \frak J(\Q)_{>0}}\textbf{w}_B(g_\f,F)W^{(\ell)}_B(g_\infty),\ g=g_\f\cdot g_\infty\in \bG(\A_\Q)$$
which is absolutely and uniformly convergent on any compact neighborhood of $g$. Here $\textbf{w}_B(g_\f,F)$ is 
just a $\C$-valued function on $\bG(\A_\f)$ attached to each $B$ but its relation to the Whittaker functionals will be given as 
follows: 
We denote by $\mathcal{A}_\ell(\bold P(\Q)\bs \bG(\A_Q))$ the space of all smooth functions on $\bold P(\Q)\bs \bG(\A_Q)$ 
which admit the holomorphic Fourier expansion of the weight $\ell$.  

\begin{lemma}\label{rel-fc}
If $I:A(\mu_\f)\lra \mathcal{A}_\ell(\bold P(\Q)\bs \bG(\A_Q))$ is an intertwining operator, then there exists 
a complex number $c_B$ for each $B\in \frak J(\Q)_{>0}$ such that 
$$I(\phi)(g)=\sum_{B\in \frak J(\Q)_{>0}}c_B\cdot w^{\mu_\f}_B(g_\f\cdot \phi)W^{(\ell)}_B(g_\infty)
$$
for all $\phi\in A(\mu_\f)$, $g=g_\f\cdot g_\infty\in \bG(\A_\Q)=\bG(\A_\f)\bG(\R)$ and 
\begin{equation}\label{cB-rel}
c_{m\cdot B}=\nu_\f(\nu(m))^{-1}\sgn(a)^\ell c_B
\end{equation}
for $m\in \bold M(\Q)$. Conversely for complex numbers $\{c_B\}_{B\in \frak J(\Q)_{>0}}$ satisfying $(\ref{cB-rel})$, if the formal series defined by 
$$\mathcal{F}(g;\phi):=\sum_{B\in \frak J(\Q)_{>0}}c_B\cdot w^{\mu_\f}_B(g_\f\cdot \phi)W^{(\ell)}_B(g_\infty),\ g\in \bold G(\A_\Q)
$$
converges absolutely for any $\phi\in A(\mu_\f)$, then $\mathcal{F}(g;\ast)$ defines an intertwining operator from 
$A(\mu_\f)$ to $\mathcal{A}_\ell(\bold P(\Q)\bs \bG(\A_Q))$.  

\end{lemma}
\begin{proof} The converse is obvious. For each $F\in \mathcal{A}_\ell(\bold P(\Q)\bs \bG(\A_Q))$ there exists a compact open subgroup $U$ of $\frak J(\A_\f)$ 
such that 
$$F(p_{zu}g)=F(p_{z}g),\ z\in \frak J(\A_\Q),\ u\in U,\ {\rm and}\ g\in \bG(\A_\Q)
$$
by the smoothness of $F$. Since $F$ is also left $\bold N(\Q)$-invariant and $\frak J(\Q)\bs U=\frak J(\Q)\bs \frak J(\A_\f)$, 
we have that $F(p_{z}g)=F(p_{z_\infty}g)$ for $z\in \frak J(\A_\Q)$ and $g\in\bG(\A_\Q)$.  
By using this and (\ref{whi}) it follows for $F\in \mathcal{A}_\ell(\bold P(\Q)\bs \bG(\A_Q))$ that  
\begin{eqnarray}\label{fw}
W_B(g_\f,F)&=&\ds\int_{\bold N(\Q)\bs \bold N(\A_\Q)}F(p_z\cdot g_\f)\overline{\psi_B(z)}dz=
\ds\int_{\bold N(\Q)\bs \bold N(\A_\Q)}F(p_{z_\infty}\cdot g_\f)\overline{\psi_B(z)}dz \nonumber \\ 
&=&\ds\int_{\bold N(\Q)\bs \bold N(\A_\Q)}\Big(
\sum_{B'\in \frak J(\Q)_{>0}}\textbf{w}_{B'}(g_\f,F)W^{(\ell)}_{B'}(p_{z_\infty})\Big)\overline{\psi_B(z)}dz \\
&=& \ds\int_{\bold N(\Q)\bs \bold N(\A_\Q)}\Big(
\sum_{B'\in \frak J(\Q)_{>0}}\textbf{w}_{B'}(g_\f,F)\det(B')^{\frac{\ell}{2}} \textbf{e}((B',\textbf{i}))
\textbf{e}((B',z_\infty)))\Big)\overline{\psi_{B}(z)}dz \nonumber \\
&=& \textbf{w}_B(g_\f,F)\det(B)^{\frac{\ell}{2}} \textbf{e}((B,E\sqrt{-1})) \nonumber. 
\end{eqnarray}
In particular, for $z_\f\in \frak J(\A_\f)$ 
$$\textbf{w}_B(p_{z_\f}g_\f,F)=\psi_{B,\f}(z_\f)\textbf{w}_B(g_\f,F).$$ 
Therefore the functional $\phi\mapsto \textbf{w}_B(1,I(\phi))$ is an element in ${\rm Wh}_B(A(\mu_\f))$. 
By Proposition \ref{whi-glo} there exists a complex number $c_B$ such that $\textbf{w}_B(1_\f,I(\phi))=c_B w^{\mu_\f}_B(\phi)$ 
and we have $\textbf{w}_B(g_\f,I(\phi))=\textbf{w}_B(1_\f,I(g_\f\cdot \phi))=c_B w^{\mu_\f}_B(g_\f \cdot \phi)$. 
The last claim follows from the left invariance by $\bold P(\Q)$ with (\ref{whi}) and Lemma \ref{degp}-(2). 
\end{proof}

\section{Automorphic forms on $SL_2(\A_\Q)$}\label{sl2} 
For a character $\mu_p$, let $I_1(s,\mu_p)=Ind_{B(\Bbb Q_p)}^{SL_2(\Bbb Q_p)}\, \mu_p |\cdot |^s$ be the principal series.
When $\mu_p=1$ and $s=1$, it has a unique subrepresentation $A_1(|\cdot|_p)$. If $\mu_p$ is a quadratic character, $I_1(0,\mu_p)$ is a sum of two irreducible representations.
Let $\mu$ be a unitary character of $\Q^\times\backslash\A^\times$. 
Let $\mathbb{K}=SL_2(\widehat{\Z})\times SO_2(\Bbb R)$ be the standard maximal compact subgroup of $SL_2(\A)$. 
We define the space $I_1(s,\mu)$ consisting of any smooth, $\mathbb{K}$-finite 
function $\phi:{\rm SL}_2(\A)\lra \C$ such that 
$$\phi(pg)=\delta^{\frac{1}{2}}_B(p)|a|^s_{\A}\mu(\nu(p))\phi(g),
$$
 for any $p=
\left(\begin{array}{cc}
a & b \\
0 & a^{-1} 
\end{array}
\right)\in B(\A)$ and any  $g\in {\rm SL}_2(\A)$. 
Here $B$ is the Borel subgroup of ${\rm SL}_2$ which consists of upper-triangular matrices and 
$\delta^{\frac{1}{2}}_B(p)=|a|_\A$. We factor it as $I_1(s,\mu)\simeq \otimes'_p I_{1}(s,\mu_p)$ for $\mu=\otimes'_p \mu_p$. 
We often drop ``$p$" from the notation $I_{1}(s,\mu_p)$ when we are in a local situation. 

Let $\ell$ be a positive integer, and let $f$ be a newform in $S^{\rm new}_\ell(\mathbb{H})^{{\rm ns}}$ with the central character 
$\chi$. [In classical language, it belongs to $S_\ell(\Gamma_0(N),\chi)\subset S_\ell(\G_1(N))$ for some positive integer $N$ and a finite order character $\chi:\G_0(N)\lra \C^\times, 
\left(\begin{array}{cc}
a & b\\
c & d 
\end{array}
\right)\mapsto \chi(d)
$.] 
We need to assume 
\begin{equation}\label{parity}
\chi(-1)=(-1)^\ell
\end{equation}
for $f$ to exist.
To such a $f$, we can associate 
an automorphic form $\phi_f:SL_2(\Q)\bs SL_2(\A_\Q)\lra \C$ by 
\begin{equation}\label{cl-ad}
\phi_f( \gamma\cdot u\cdot h_\infty)=f(h_\infty \sqrt{-1})j_{{\rm SL}_2}(h_\infty,\sqrt{-1})^{-\ell}
\end{equation}
for  $\gamma\in {\rm SL}_2(\Q)$, $u\in K_1(N)$, and $h_\infty\in {\rm SL}_2(\R)$ where $K_1(N)$ is the subgroup of $SL_2(\widehat{\Z})$ consisting of the elements 
so that they are congruent to $\left(\begin{array}{cc}
1 & \ast\\
0 & 1 
\end{array}
\right)$ modulo $N$. 
Let $\pi'=\otimes'_p\pi'_p$ be the automorphic representation of $GL_2(\A_\Q)$ associated to $f$. Let $S_f$ be the set of all finite places $p$ so that $\pi'_p={\rm St}_{GL_2}$ or its twist. Then by assumption, if $p\notin S_f\cup\{\infty\}$, up to twist 
$\pi'_p\simeq \pi(\mu_{1p},\mu_{2p})$
for some (unitary) characters $\mu_{ip}:\Q^\times_p \lra \C^\times,\ i=1,2$ so that either of $\mu_{ip},i=1,2$ is unramified (cf. Proposition 2.8 of \cite{LW}). 

Let $\pi$ be the cuspidal automorphic representation of ${\rm SL}_2(\A_\Q)$ associated to $\phi_f$. 
We factorize it as $\pi=\pi_\f\otimes \pi_\infty=\otimes'_{p}\pi_p$. 
Then we see that 
$$\pi_\f\simeq \bigotimes_{p\not\in S_{\bf{f}}\cup\{\infty\}} \pi_p\otimes \bigotimes_{p\in S_{\bf{f}}} A_1(|\cdot|)=:A_1(\pi_\f)
$$ 
where $\mu_p:=\mu_{1p}\mu^{-1}_{2p}$, and $\pi_p$ is an irreducible direct summand of $I_1(0,\mu_p)$.
Define $\mu_\f:\A^\times_\f\lra \C^\times$ by 
$$\mu_\f=\bigotimes_{p\not\in S_f\cup\{\infty\}} \mu_p\otimes \bigotimes_{p\in S_f}|\cdot|_p.
$$
The Whittaker-Fourier expansion of $\phi_f$ takes a form 
\begin{equation}\label{ct}
\phi_f(h)=\sum_{r\in \Q_{>0}}c_r w^{\mu_\f}_r(h_\f (\phi_f)_\f)W^{(\ell)}_r(h_\infty),
\end{equation}
where $c_r=c_n(f)$ when $r=n$, which we will show below.
Since $\phi_f(h)$ is left invariant by the action of ${\rm SL}_2(\Q)$, by (\ref{sl2whi}) and (\ref{sl2whip}), we have a property 
which plays an important role: 
\begin{equation}\label{rel-imp}
c_{a^2r}=c_r \mu_\f(a)^{-1}\sgn(a)^\ell
\end{equation}
for all $a$, $r\in \Q_{>0}$.   
The coefficients $\{c_r\}_{r\in \Q_{>0}}$ have the following distinguished properties:
\begin{prop}\label{dist} For any $\Phi\in A_1(\pi_\f)$ and  $\eta\in \Q_{>0}$, the series  
$$F(\Phi,\{c_{\eta t}\}_t):=\sum_{r\in \Q_{>0}}c_{\eta r}  w^{\mu_\f}_r(h_\f \cdot\Phi)W^{(\ell)}_r(h_\infty)
$$
is absolutely convergent and it defines a cuspidal automorphic form on $SL_2(\A_\Q)$. 
\end{prop}
\begin{proof} The convergence follows from the Whittaker-Fourier expansion and the estimates of $c_{\eta r}$.

In a usual way we can associate to $f$, a cuspidal automorphic form $\varphi_f$ of ${\rm GL}_2(\A_\Q)$ so that 
its restriction to ${\rm SL}_2(\A_\Q)$ coincides with $\phi_f$. 
Then it is easy to see that the restriction of $(\diag(\eta^{-1},1))\varphi_f$ to ${\rm SL}_2(\A_\Q)$ becomes 
$$\phi_{f,\eta}(h):=\ds\sum_{r\in \Q_{>0}}c_{r}\mu_\f(\eta)^{-1}\eta^{\frac{\ell}{2}} w^{\mu_\f}_{\eta^{-1} r}(h_\f (\phi_f)_\f)
W^{(\ell)}_{\eta^{-1} r}(h_\infty).$$ 
From this we have 
\begin{eqnarray}
\phi_{f,\eta}(h)&=& 
\sum_{r\in \Q_{>0}}c_{\eta r}\mu_\f(\eta)^{-1}\eta^{\frac{\ell}{2}} w^{\mu_\f}_{r}(h_\f (\phi_f)_\f)
W^{(\ell)}_{r}(h_\infty)\nonumber  \\
&=&\mu_\f(\eta)^{-1}\eta^{\frac{\ell}{2}}F(\Phi,\{c_{\eta r}\}_r).\nonumber
\end{eqnarray}
Therefore $F(\Phi,\{c_{\eta r}\}_r)$ is automorphic. 
The series can be written 
as a linear combination of the left translations of $\phi_{f,\eta}$ since 
$(\phi_f)_\f$ generates $A_1(\pi_\f)$. Hence we have the claim.   
\end{proof}
To end this section we will explicitly compute $\{c_t\}_{t\in \Q_{>0}}$ for $f$. Let us keep the notations as above. 
Let $f=q+\ds\sum_{n=2}^{\infty} a_f(n)q^n$ be the normalized Fourier expansion. 
We factorize $\varphi_f=\otimes'_p \varphi_p$ as an automorphic form on $GL_2(\A_\Q)$. 
The first task we have to do is to make explicit the relationship between $\varphi_p$ and a unique local newform $f^{{\rm new}}_p$
of $\pi'_p$ defined in \cite{Schmidt}. 

By (\ref{rel-imp}) and the multiplicative property, we have to compute only 
$c_p(f)$ for each prime $p$. We can obtain the necessary other information by using Hecke operators and epsilon factors. 

Let us first consider the case when $\pi'_p=\pi(\mu_{1p},\mu_{2p})$ is unramified. Put $\alpha=\mu_{1p}(p^{-1})$ and $\beta=\mu_{2p}(p^{-1})$ for simplicity. Therefore 
$\mu_p(p)=\ds\frac{\beta}{\alpha}$.  
Note that the convention for the Satake parameters is slightly different from the one in \cite{Schmidt} but we may simply  
replace $(\alpha,\beta)$ with $(\beta,\alpha)$. 

Then by (\ref{cl-ad}) and the observation in Section 8 of \cite{MY} for the relation between 
the global Hecke operators and the local Hecke operators, we see that 
$\varphi_p=f^{{\rm new}}_p$.
By (\ref{cn}),
$a_f(p)=p^{\ell}c_p(f)w^{\mu_p}_p(f^{{\rm new}}_p).
$
Assume that $\alpha\not=-\beta$. Then by Lemma 2.2.1 of \cite{Schmidt} (applied for the additive character $\psi=\psi_p$ with the conductor $c(\psi)=-1$) and (\ref{cn}), we see that 
$$w^{\mu_p}_p(f^{{\rm new}}_p)=p^{-\frac{1}{2}}(1-\ds\frac{\beta}{\alpha}p^{-1})^{-1}(1-\ds\frac{\beta}{\alpha}p^{-1})
(1+\frac{\beta}{\alpha})=p^{-\frac{1}{2}}(1+\frac{\beta}{\alpha}).
$$ 
Since $a_f(p)=p^{\frac{\ell-1}{2}}(\alpha+\beta)$, we have 
$c_p(f)=\alpha p^{\frac{\ell}{2}}$.

When $\alpha=-\beta$, $a_f(p)=0$, and we cannot apply the same method since both sides are zero. 
Instead of local newforms, we use the $p$-stabilized form given by 
$f_{\alpha}(\tau)=f(\tau)-p^{\frac{\ell-1}{2}}\beta f(p\tau)$ (cf. Section 9.1 of \cite{MY}). Then its local component $f_{\alpha,p}$ becomes, in terms of the newform,  
$f_{\alpha,p}=f^{{\rm new}}_p-p^{-\frac{1}{2}}\beta \pi'({\rm diag}(1,p^{-1}))f^{{\rm new}}_p$. 
Then we have 
$$a_f(p)-p^{\frac{\ell-1}{2}}\beta =-p^{\frac{\ell-1}{2}}\beta=p^\ell c_p(f) w^{\mu_p}_p(f_{\alpha,p})
$$ 
since $a_f(p)=0$. By Lemma 2.2.1 of \cite{Schmidt} again, we compute 
\begin{eqnarray}
w^{\mu_p}_p(f_{\alpha,p})&=& p^{-\frac{1}{2}}(1-\ds\frac{\alpha}{\beta}p^{-1})^{-1}\int_{\Q_p}f_{\alpha,p}(\left(\begin{array}{cc}
0 & 1\\
-1 & -z 
\end{array}
\right))\overline{\psi(pz)}dz  \nonumber  \\
&=& -p^{-\frac{1}{2}}(1-\ds\frac{\alpha}{\beta}p^{-1})^{-1}p^{-\frac{1}{2}}\beta\int_{\Q_p}f(\left(\begin{array}{cc}
0 & p^{-1}\\
-1 & -p^{-1}z 
\end{array}
\right))\overline{\psi(pz)}dz \mbox{\ (since $\alpha+\beta=0$)} \nonumber\\
&=& -p^{-\frac{1}{2}}(1-\ds\frac{\alpha}{\beta}p^{-1})^{-1}\alpha\beta\int_{\Q_p}f(\left(\begin{array}{cc}
0 & 1\\
-1 & -z 
\end{array}
\right))\overline{\psi(p^2z)}dz  \nonumber \\
&=&  -p^{-\frac{1}{2}}(1-\ds\frac{\alpha}{\beta}p^{-1})^{-1}\alpha\beta (1-\ds\frac{\alpha}{\beta}p^{-1})(1+\frac{\alpha}{\beta}+\frac{\alpha^2}{\beta^2})=\beta^2 p^{-\frac{1}{2}}.  \nonumber
\end{eqnarray}
Summing up, we have the same result 
$c_p(f)=\alpha p^{-\frac{\ell}{2}}.$

When $\mu_{1p}$ is unramified and $\mu_{2p}$ is ramified, then the similar computation shows that 
$c_p(f)=\ds\frac{\beta^{-1}p^{-\frac{\ell+1}{2}}}{1-p^{-1}}\mu_{2p}(-1)$.
When $\mu_{1p}$ is ramified and $\mu_{2p}$ is unramified, then 
$c_p(f)=\beta p^{-\frac{\ell}{2}}$. Note that in either case, $L(1,\mu_p)=1$ and we have $|c_p(f)|\le p^{-\frac{\ell}{2}}$. 

Next we consider the case when $\pi_p'$ is a twist of the Steinberg representation. 
Since we will factor through $SL_2$, in our purpose we may assume that $\pi_p'$ is a unique submodule of $\pi(|\cdot|^{\frac{1}{2}},|\cdot|^{-\frac{1}{2}})$. 
Then by Atkin-Lehner theory, $a_f(p)=-\varepsilon_p p^{\frac{\ell}{2}-1}$ where $\varepsilon_p$ is the eigenvalue of the 
Atkin-Lehner involution $W_p$. Then by (37) of \cite{Schmidt}, we have 
$c_p(f)=\varepsilon_p p^{-\frac{\ell-1}{2}}(1-p^{-1})$ and it yields $|c_p(f)|\le p^{-\frac{\ell-1}{2}}$. 
Hence we have proved that 
\begin{prop}\label{estct}Keep the notations as above. 
Then $c_n(f)=0$ if $n\not\in \Z_{>0}$ and otherwise $|c_n(f)|\le m^{-\frac{\ell-1}{2}}$ 
where $m$ is the product of all primes $p|n$ so that ${\rm ord}_p(n)$ is odd. 
\end{prop}

\section{Fourier-Jacobi expansions}
As seen in Section 5 of \cite{KY}, 
for each $S\in \frak J_2(\Z)_{+}$, the $S$-th Fourier-Jacobi coefficient of 
a modular form $F$ on $\frak T$ is represented by a finite sum of the products of theta series and elliptic modular forms. 
An idea of Ikeda and Yamana is to interpret this computation in terms of adelic language and check the compatibility of 
local computations and a global computation. 
Recall from Lemma \ref{useful} that 
$X\oplus r=\left(\begin{array}{cc}
X & 0 \\
0 & r
\end{array}
\right)$
and 
$R_{r,x}=v_1(x,0,0)\cdot (X\oplus r)=\left(\begin{array}{cc}
X & Xx \\
{}^t\overline{x}X & r+\sigma_R(x,x)
\end{array}
\right)$.

\subsection{The non-archimedean case} Let $p$ be a rational prime number and $\mu$ a character of $\Q^\times_p$. Recall 
the Weil representation $\omega^\psi_S$ for $S\in R_3(\Q_p)$, and the fixed non-trivial additive character $\psi:\Q_p\lra \C^\times$ 
acting on the Schwartz space $\mathcal{S}(X(\Q_p))$, where $X=\{v_1(x,0,0)\ |\ x\in \frak C^2\}$ is given by 
the new coordinate $v_1$. 

The following is a local analogue of Theorem 7.1 of \cite{KY}: 
\begin{lemma}\label{longeq} For $\phi\in I(s,\mu)$, $\Phi\in \mathcal{S}(X(\Q_p))$ and $h\in \bold H(\Q_p)$, the integral 
$$\beta^\psi_S(h;\phi\otimes \overline{\Phi}):=\prod_{i=0}^2 L(9-4i+s,\mu)\int_{Y(\Q_p)}\int_{Z(\Q_p)}
\phi(\iota\cdot \iota_{e_3}v_1(y,0,z)h)
\overline{(\omega^\psi_S(v_1(y,0,z)h)\Phi)(0)}dydz
$$
is absolutely convergent if $Re(s)>-1$ and it gives an $\bold V(\Q_p)$-invariant and ${\rm SL}_2(\Q_p)$-intertwining map 
$$\beta^\psi_S:I(s,\mu)\otimes \mathcal{S}(X(\Q_p)) \lra I_1(s,\mu).
$$
Hence $\beta^\psi_S(v_1\cdot h\cdot \gamma; \phi\otimes \overline{\Phi})=\gamma\cdot \beta^\psi_S(h; \phi\otimes \overline{\Phi})$ 
for any $\gamma\in {\rm SL}_2(\Q_p)$ and $v_1\in \bold V(\Q_p)$.  
\end{lemma}  
\begin{proof} By the computation in p.243 of \cite{KY} (but a modification is necessary and see the remark below), we have 
$$\int_{Y(\Q_p)}\int_{Z(\Q_p)}
\phi(\iota\cdot \iota_{e_3}v_1(y,0,z)h)
\overline{(\omega^\psi_S(v_1(y,0,z)h)\Phi)(0)}dydz=
\int_{\bold V(\Q_p)} \phi(\iota\cdot \iota_{e_3}v_1 h)\overline{(\omega^\psi_S(\iota_{e_3}v_1 h)\Phi)(0)}dv_1.
$$
The invariance for ${\rm SL}_2(\Q_p)$ and $\bold V(\Q_p)$ is clear from this formula. The convergence follows from 
the smoothness of $f$. The condition ``$Re(s)>-1$" is necessary to make sense of the normalizing factors in front of the integral. 
By a similar computation done in page 243 of \cite{KY} we can check the map takes the image in $I_1(u,\mu)$. 
\end{proof}

For $t\in \Q^\times_p$ and $\phi\in I_1(s,\mu)$, define 
\begin{equation}\label{sl2whip}
w^{\mu,s}_t(\phi)=|t|^{\frac{1}{2}}_p L(1+s,\mu)\int_{z'\in \Q_p} \phi(i(\iota_{e_3})n(z'))\overline{\psi(tz')}dz'
\end{equation}
via the isomorphism $i: \bold H(\Q_p)\simeq {\rm SL}_2(\Q_p)$. 

\begin{lemma}\label{com1} For $\phi\in I(s,\mu), t\in \Q^\times_p$, and $\Phi\in \mathcal{S}(X(\Q_p))$, the equality  
$$w^{\mu,s}_t(\beta^\psi_S(\ast; \phi\otimes \overline{\Phi}))=C(S)|t|^{-4}_p \int_{X(\Q_p)}
\overline{\Phi(x)}w^{\mu,s}_{S\oplus t}(v_1(x,0,0)\cdot \phi)dx
$$
holds, where $C(S)=|\det(S)|^{-\frac{9}{2}}_p$.
\end{lemma}
\begin{proof} A direct computation shows 
\begin{eqnarray}\label{FJ}
&& \left(|t|^{\frac{1}{2}}_p\prod_{i=0}^2 L(9-4i+s,\mu)\right)^{-1} w^{\mu,s}_t(\beta^\psi_S(\ast;\phi\otimes \overline{\Phi})) \\ 
&&\phantom{xxxx} =\left(\prod_{i=0}^2 L(9-4i+s,\mu) \right)^{-1}\int_{z'\in \Q_p}
\beta^\psi_S(i(\iota_{e_3})n(z'); \phi\otimes \overline{\Phi})\overline{\psi(tz')}dz' \nonumber \\
&&\phantom{xxxx} =
\int_{z'\in \Q_p}\int_{\bold V(\Q_p)} \phi(\iota \cdot v_1\cdot n(z'))
\overline{(\omega^\psi_S(v_1\cdot n(z'))\Phi)(0)}\psi(-tz')dv_1dz' \nonumber \\
&&\phantom{xxxx} =
\int_{z'\in \Q_p}\int_{\bold V(\Q_p)} \phi(\iota \cdot v_1\cdot n(z'))
\overline{(\omega^\psi_S(v_1))\Phi)(0)}\psi(-tz')dv_1dz' \nonumber
\end{eqnarray}
By transforming $v_1$ with $n(z')v_1n(z')^{-1}$, 
\begin{eqnarray*}
(\ref{FJ})&=&\int_{z'\in \Q_p}\int_{\bold V(\Q_p)} \phi(\iota\cdot n(z')\cdot v_1)
\overline{(\omega^\psi_S(n(z')v_1))\Phi)(0)}\psi(-tz')dv_1dz'\nonumber \\
&=& 
\int_{Z\in Z(\Q_p)}\int_{\bold V(\Q_p)} \phi(\iota\cdot n(z')\cdot v_1(0,y,z-\sigma(x,y))v_1(x,0,0))
\overline{\Phi(x)\psi_S(z-\sigma(x,y)))}\psi(-tz')dv_1dz'\nonumber \\
&=& 
\int_{Z\in Z(\Q_p)}\int_{\bold V(\Q_p)} \phi(\iota\cdot n(z')\cdot v_1(0,y,z)v_1(x,0,0))
\overline{\Phi(x)\psi_S(z))}\psi(-tz')dv_1dz'\nonumber \\
&=& 
\int_{Z\in Z(\Q_p)}\int_{X(\Q_p)} \phi(\iota\cdot p_Z \cdot v_1(x,0,0) )
\overline{\Phi(x)\psi_{S\oplus t}(Z)}dxdZ \nonumber \\
&=& 
\int_{Z\in Z(\Q_p)}\int_{X(\Q_p)} \phi(\iota\cdot p_Z \cdot v_1(x,0,0) )
\overline{\Phi(x)\psi_{S\oplus t}(Z)}dxdZ. \nonumber
\end{eqnarray*}

Since this integral is absolutely convergent, we can change the order of the double integral. Hence  
\begin{eqnarray}
(\ref{FJ})&=& \int_{X(\Q_p)} \overline{\Phi(x)}\int_{Z\in Z(\Q_p)} \phi(\iota\cdot p_Z \cdot v_1(x,0,0) )
\overline{\psi_{S\oplus t}(Z)}dZdx. \nonumber \\
&=& |\det(S\oplus t)|^{-\frac{9}{2}}(\prod_{i=0}^2 L(s-4i+s,\mu))^{-1}
\int_{X(\Q_p)}\overline{\Phi(x)}w^{\mu,u}_{S\oplus t}(v_1(x,0,0)\phi)dx. \nonumber
\end{eqnarray}
Clearing up the factors we have the claim. 
\end{proof} 

\subsection{The archimedean case} For $R\in \jt(\R)_{>0}$, define $\Phi_R\in \mathcal{S}(X(\R))$ by 
$\Phi_R(x)=\textbf{e}(-2\pi \sigma_R(x,x))$. Recall the degenerate Whittaker functions $W^{(\ell)}_{R\oplus r},\ W^{(\ell)}_r$ in 
Section \ref{dwa}. 

\begin{lemma}\label{awc}For $R\in \jt(\R)_{>0},\ r \in \R_{>0}$, $h\in {\rm SL}_2(\R)$ and $\ell\in \Z_{\ge 0}$, the equality 
$$\int_{X(\R)}W^{(\ell)}_{R\oplus r}(v_1(x,0,0)h)\overline{\omega^\psi_R(h)\Phi_R(x)}dx=C(R)r^4 W^{(\ell-8)}_r
(h),
$$
where $C(R)=\det(R)^{\ell-4}2^{-16}$. 
\end{lemma} 
\begin{proof} It easy to check that $\omega^\psi_R(h)\Phi_R(x)=j_{{\rm SL}_2}(h,\sqrt{-1})^8\textbf{e}(\sigma_R(x,x)\tau)$ 
for $\tau:=h\sqrt{-1},\ h\in {\rm SL_2}(\R)$, and $x\in X(\R)$. 
Note that $(R_{r,x},h(E\sqrt{-1}))=\sqrt{-1}{\rm Tr}(R)+r\tau+\sigma_R(x,x)\tau$. 
Then by Lemma \ref{useful}, (\ref{autop}) and (\ref{whi}) a direct computation shows 
\begin{eqnarray}
&&\int_{X(\R)}W^{(\ell)}_{R\oplus r}(v_1(x,0,0)h)\overline{\omega^\psi_R(h)\Phi_R(x)}dx \nonumber \\
&=&  \det(R_{r,x})^{\frac{\ell}{2}} \textbf{e}((R_{r,x},h(E\sqrt{-1})))j(h,E\sqrt{-1})^{-\ell}\overline{j_{{\rm SL}_2}(h,\sqrt{-1})^8
\textbf{e}(\sigma_R(x,x)\tau)}dx \nonumber \\
&=& \det(R)^{\frac{\ell}{2}} r^{\frac{\ell}{2}} \textbf{e}(r\tau)j_{{\rm SL}_2}(h,\sqrt{-1})^{-(\ell-8)}\frac{1}{({\rm Im}\tau)^8}
\int_{X(\R)}e^{-4\pi \sigma_R(x,x)\tau}dx \nonumber \\
&=& \det(R)^{\frac{\ell-8}{2}}2^{-16}r^4 W^{(\ell-8)}_r(h)\nonumber
\end{eqnarray}
where the last equality follows from the formula 
$\ds\int_{X(\R)}e^{-4\pi \sigma_R(x,x)\tau}dx =2^{-16}\det(R)^{-4}({\rm Im}\, \tau)^8$ 
\newline (see the line -2 in page 252 of \cite{KY}).  
\end{proof}

\subsection{The global case} 
Let $\psi=\otimes'_{p}\psi_p=\psi_\f\otimes \psi_\infty:\A\lra \C^\times$ be the standard non-trivial additive character so that $\psi_\infty=\textbf{e}_\infty$. For $S\in \jt(\Q)_{>0}$ we define the (finite) global Weil representation by 
$\omega_S=\otimes'_{p}\omega_{S,p}=\omega_{S,\f}\otimes \omega_{S,\infty}$ acting on $\mathcal{S}(X(\A_\Q))$. 
For $\Phi\in \mathcal{S}(X(\A_\f))$ we define 
$$\Phi_S(x)=\Phi(x_\f)\Phi_S(x_\infty),\ \Phi_S(x_\infty)=\textbf{e}(-2\pi \sigma_R(x_\infty,x_\infty))
$$
for $x=x_\f\cdot x_\infty\in X(\A_\Q)$. 

For $F\in \mathcal{A}_\ell(\bold P(\Q)\bs \bold G(\A_\Q))$ and take its holomorphic Fourier expansion of weight $\ell$ as 
$$F(g)=\sum_{B\in \frak J(\Q)_{>0}}\textbf{w}_B(g_\f,F)W^{(\ell)}_B(g_\infty),\ g=g_\f\cdot g_\infty\in \bG(\A_\Q).$$
For $\Phi\in \mathcal{S}(X(\A_\f))$ we define the $(S,\phi)$-th Fourier-Jacobi coefficient by 
$$F^S_\Phi(h):=\int_{\bold V(\Q)\bs \bold V(\A_\Q)}F(v_1h)\overline{\Theta^{\psi_S}(\omega_S(v_1h)\Phi_S)}dv_1, \ h \in 
{\rm SL}_2(\A_\Q).
$$  

The following lemma is an analogue of the argument at line 5-6, p.246 in \cite{KY}:
\begin{lemma}\label{important1} Keep the notation as above. Then for $F\in \mathcal{A}_\ell(\bold P(\Q)\bs \bold G(\A_\Q))$, 
\begin{enumerate}
\item $F^S_\Phi(h)=C_1(S)\ds\sum_{r\in \Q_{>0}} r^4 W^{\ell-8}_r(h_\infty)
\Big(\int_{X(\A_\f)}\textbf{w}_{S\oplus r}(v_1(x,0,0)h_\f,F)\overline{(\omega_{S,\f}(h_\f)\Phi)(x_\f)}dx_\f\Big)$,  
\newline where $C_1(S)=\det(S)^{\frac{\ell-8}{2}}2^{-16}$.  

\item $F$ is left invariant by $\bold G(\Q)$ if and only if $(g_\f\cdot F)^S_\Phi$ is left invariant by ${\rm SL}_2(\Q)$ 
for any $g_\f\in \bold G(\A_\f)$, $S\in \jt(\Q)_{>0}$, and $\Phi\in \mathcal{S}(X(\A_\f))$. 
\end{enumerate}
\end{lemma}
\begin{proof} By the definition 
$$F^S_\Phi(h)=\int_{\bold V(\Q)\bs \bold V(\A_\Q)}\sum_{B\in \frak J(\Q)_{>0}}
\textbf{w}_B(g_\f,F)W^{(\ell)}_B(g_\infty)\overline{\Theta^{\psi_S}(\omega_S(v_1h)\Phi_S)} dv_1.
$$
Observing the integral at the archimedean place we see that the sum inside the integral runs over 
$B\in\frak J(\Q)_{>0}$ so that 
$B=\left(\begin{array}{cc}
S & \ast \\
{}^t \overline{\ast} & \ast
\end{array}
\right)
$. Further such a $B$ can be written as 
$$B=v_1(\lambda,0,0)(S\oplus r)=S_{r,\lambda}
$$
for some $\lambda\in X(\Q)$ and $r\in \Q_{>0}$. By (\ref{fw}) one can check that
$$W_B(g,F)=W_B(g_\f,F)W^{(\ell)}_B(g_\infty)$$
and hence we have $$\textbf{w}(g_\f,F)W^{(\ell)}_{B}(g_\infty)=\det(B)^{-\frac{\ell}{2}}\textbf{e}((B,E\sqrt{-1}))^{-1}W_B(g,F).$$
For $B=S_{r,\lambda}$ we see that 
$\det(B)^{-\frac{\ell}{2}}\textbf{e}((B,E\sqrt{-1}))^{-1}=
\det(S\oplus r)^{-\frac{\ell}{2}}\textbf{e}((S\oplus r,E\sqrt{-1}))^{-1}:=C_2(S,r)$ which is 
independent of $\lambda$.  

By Lemma \ref{useful}, we have    
\begin{eqnarray}\label{wc}
W_B(v_1(x,y,z)h,F)&=& W_B(v_1(0,y,z+\sigma(x,y))v_1(x,0,0)h,F) \nonumber \\
&=& \psi_S(z+\sigma(x,y))\psi_S(2\sigma(\lambda,y))W_B(v_1(x,0,0)h,F) \nonumber
\end{eqnarray}
for $v_1=v_1(x,y,z)\in \bold V(\A_\Q)$ and $h\in {\rm SL}_2(\A_\Q)$. 
On the other hand we see easily that 
$$W_B(v_1(x,0,0)h,F)=\int_{Z\in \frak J(\Q)\bs \frak J(\A_\Q)}F(p_Z v_1(x+\lambda,0,0))\overline{\psi_{S\oplus r}(Z)}dZ.
$$
Therefore 
\begin{eqnarray}\label{wc1}
&&\int_{\bold V(\Q)\bs \bold V(\A_\Q)}F(v_1(x,y,z)h)\overline{\Theta^{\psi_S}(\omega_S(v_1h)\Phi_S)}dv_1 \nonumber \\
&=& \int_{\bold V(\Q)\bs \bold V(\A_\Q)} \sum_{B=S_{\lambda,r}\in\frak J(\Q)_{>0}}C_2(S,r)^{-1}
\psi_S(z)\psi_S(2\sigma(\lambda,y))\psi_S(\sigma(x,y))W_{S\oplus r}(v_1(x+\lambda,0,0)h) \nonumber \\
&\times & \overline{\sum_{\xi\in X(\Q)}\omega_S(h)\Phi_S(\xi+x)\psi_S(z)\psi_S(2\sigma(\xi,y))\psi_S(\sigma(x,y)) }
dv_1. \nonumber
\end{eqnarray}
Since the integral of $\psi_S(2\sigma(\lambda-\xi,y))$ over $Y(\Q)\bs Y(\A_\Q)$ is zero unless $\xi=\lambda$. 
It follows from this that 
\begin{eqnarray}\label{wc1}
&& \int_{\bold V(\Q)\bs \bold V(\A_\Q)}F(v_1(x,y,z)h)\overline{\Theta^{\psi_S}(\omega_S(v_1h)\Phi_S)}dv_1 \nonumber \\
&=&\int_{X(\Q)\bs X(\A_\Q)}\sum_{r\in \Q_{>0}}\sum_{\lambda\in X(\Q)}
C_2(S,r)^{-1}W_{S\oplus r}
(v_1(\lambda+x,0,0))\overline{\omega_S(h)\Phi(x+\lambda)} dx \nonumber \\
&=&\int_{X(\A_\Q)}\sum_{r\in \Q_{>0}}C_2(S,r)^{-1}W_{S\oplus r}
(v_1(x,0,0))\overline{\omega_S(h)\Phi(x)} dx \nonumber \\
&=&\sum_{r\in \Q_{>0}}\int_{X(\R)}W_{S\oplus r}
(v_1(x_\infty,0,0))\overline{(\omega_{S,\infty}(h_\infty)\Phi_S)(x_\infty)} dx_\infty  \nonumber \\
&\times & 
\int_{X(\A_\f)}\textbf{w}_{S\oplus r}
(v_1(x_\f,0,0))\overline{(\omega_{S,\f}(h_\f)\Phi)(x_\f)} dx_\f \nonumber  \\
&=& \det(S)^{\frac{\ell-8}{2}}2^{-16}\sum_{r\in \Q_{>0}} r^4 W^{(\ell-8)}_r(h_\infty)\int_{X(\A_\f)}\textbf{w}_{S\oplus r}(v_1(x_\f,0,0))
\overline{(\omega_S(h_\f)\Phi)(x_\f)}dx_\f. \nonumber
\end{eqnarray}

Next we prove the second claim: ``only if" part is clear. 
Put $w^{(\ell)}_B(g,F):=\textbf{w}_B(g_\f,F)W^{(\ell)}_B(g_\infty)$ for simplicity. 
$$F(g)=\sum_{S\in \jt(\Q)_{>0}}F^S(g),\quad F^S(g):=\sum_{r\in \Q_{>0}}\sum_{u\in X(\Q)}w^{(\ell)}_{S_{r,u}}(v_1(u,0,0)g,F).
$$
Let $C_S$ be an open compact subgroup of $\frak J(\A_\f)$ such that $F^S$ is right invariant. 
By the admissibility of $\omega_{S,\f}$, ${\rm dim}(\omega_{S,\f})^{C_S}$ is  finite. Therefore 
we can take an orthonormal basis $\{\phi^1,\ldots,\phi^k\}$ of $(\omega_{S,\f})^{C_S}$. 
Recall that $\bold Q=\bold L\bold V$, where the derived group of $\bold L$ is $\bold H\times Spin(1,9)$, $\bold H\simeq SL_2$.
For $g_1^\infty\in {\rm Spin}(1,9)(\R)$, $v_1\in \bold V(\A_\Q)$, and $h\in \bold H(\A_\Q)$, 
one can write 
$$F^S(v_1g^\infty_1 h)=c \sum_{i=1}^k \Theta(\omega_S(v_1h)\phi^i)F^S_{\phi^i}(h).
$$  
Here $c$ is a constant depending only on $g_1^\infty$ and $S$. 

By Iwasawa decomposition, $\bold G(\Bbb R)=\bold Q(\Bbb R)K_\infty$, where $K_\infty$ is a maximal compact subgroup of type $E_6$.
Hence any element $g\in \bold G(\A_\Q)$ can be written as $g=g^\infty_1 g^\infty_2 v_{1,\infty} k g_\f,\ 
g^\infty_1\in {\rm Spin}(1,9)(\R),\ g^\infty_2\in \bold H(\R)$, $v_{1,\infty}\in \bold V(\R)$, $k\in K_\infty$, and $g_\f\in \bold G(\A_\f)$.  
Note that $g^\infty_1 g^\infty_2=g^\infty_2 g^\infty_1$.
By using this we have 
$$F^S(g)=j(k,E)^{-\ell}(g_\f\cdot F)^S(g^\infty_1 g^\infty_2 v_{1,\infty})=
j(k,E)^{-\ell} c\sum_{i=1}^k \Theta(\omega_S(v_{1,\infty} g^\infty_2)\phi^i)(g_\f\cdot F)^S_{\phi^i}(g^\infty_2).$$
By the assumption $(g_\f\cdot F)^S_{\phi^i}$ is left invariant by $\bold H(\Q)\simeq {\rm SL}_2(\Q)$ and so are 
theta functions by the definition (note that theta functions are left invariant by $J(\Q)$). 
Hence $F_S$ is left invariant by $\bold H(\Q)$. Now we have
$\bold H(\Q)\bold P(\Q)=\bold G(\Q)$. (By \cite[p. 526]{B}, $\bold G(\Q)$ is generated by $\bold P(\Q)$ and $\iota$. By the identities in \cite[Lemma 5.2]{KY}, $\bold G(\Q)$ is generated by $\bold P(\Q)$ and $\iota_{e_3}$.)
Hence we have the claim.  
\end{proof}

\section{Proofs of the main theorems} 

\subsection{Proof of Theorem \ref{main-thm1} and Corollary \ref{cor1}} Let us keep the notations in Section \ref{sl2}. 
For a newform $f\in S_{k}(\mathbb{H})^{{\rm ns}},\ k\ge 2$, one can associate $\{c_r\}_{r\in \Q_{>0}}$ as in (\ref{ct}). 
For $B\in \frak J(\Q)_{>0}$, define $c_B:=c_{\det(B)}.$
By (\ref{rel-imp}), we see that 
$$c_{m\cdot B}=c_{\nu(m)^2\det(B)}
=\mu_\f(\nu(m))^{-1}\sgn(\nu(m))^{k+8} c_{\det(B)}=\mu_\f(\nu(m))^{-1}\sgn(\nu(m))^{k+8} c_B.
$$
In light of Lemma \ref{awc}, the infinity type of the Ikeda type lift is the holomorphic discrete series of lowest weight $k+8$. 
For $\phi\in I(\mu_{\bf{f}}^2)$, consider a formal series 
\begin{equation}\label{rel-imp1}
\mathcal{F}(g;\phi)=\sum_{B\in \frak J(\Q)_{>0}} c_B\cdot w^{\mu_\f}_B(g_\f\cdot \phi)W^{(k+8)}_B(g_\infty),\ g\in \bold G(\A_\Q).
\end{equation}

By Lemma \ref{fc2} and Proposition \ref{estct} this series is absolutely convergent and the relation (\ref{rel-imp1}) implies that 
$\mathcal{F}(g;)\in \mathcal{A}(\bold P(\Q)\bs \bold G(\A_\Q))$. 
By Lemmas \ref{com1} and \ref{important1}, 
for $\Phi\in \mathcal{S}(\bold X(\A_\f))$ and $S\in \jt(\Q)_{>0}$ we have 
\begin{eqnarray}
(\mathcal{F}(\ast;\phi))^S_\Phi(h)&=&
C_1(S)\det(S)^{\frac{9}{2}}\sum_{r\in \Q_{>0}}c_{S\oplus r}w^{\mu_\f}_r(\beta^\psi_S(\phi\otimes \overline{\omega_S(h_\f)\Phi})) 
W^{(k)}_r(h_\infty) \nonumber \\
&=&C_1(S)\det(S)^{\frac{9}{2}}\sum_{r\in \Q_{>0}}c_{S\oplus r}w^{\mu_\f}_r(\beta^\psi_S(\phi\otimes \overline{\omega_S(h_\f)\Phi})) 
W^{(k)}_r(h_\infty)           \nonumber  \\
&=&C_3(S)\sum_{r\in \Q_{>0}}c_{\det(S)r} w^{\mu_\f}_r(\beta^\psi_S(\phi\otimes \overline{\omega_S(h_\f)\Phi})) 
W^{(k)}_r(h_\infty)           \nonumber 
\end{eqnarray} 
where $C_3(S)=C_1(S)\det(S)^{\frac{9}{2}}=\det(S)^{\frac{k+9}{2}}2^{-16}$. By Proposition \ref{dist},  
$(\mathcal{F}(\ast;\phi))^S_\Phi(h)$ is an automorphic function on $SL_2(\A_\Q)$. Hence it is invariant by ${\rm SL}_2(\Q)$ from 
the left.  
Therefore by Lemma \ref{important1}-(2) $\mathcal{F}(g;\phi)$ is automorphic. 
Now by translating it into classical language we obtain (\ref{classical}), which is a Hecke cusp form in 
$S_{k+8}(\frak T)$. 

By definition the formal series (\ref{rel-imp1}) is a Hecke eigen if $\phi=\otimes_{p<\infty}{\phi_p}$ is a 
Hecke eigen vector 
(for example, we may choose the local sections $\phi_p$ for $p\nmid N$ defined in Section \ref{thm1.3}). 
Corollary \ref{cor1} now follows from this. 

\subsection{Proof of Theorem \ref{main-thm2}}\label{thm1.3}

Let $f=\ds\sum_{n=1}^\infty a_n(f) q^n$ be a new form in $S_{k}(\Gamma_0(N))^{{\rm new}}$ for a square free positive integer $N$. 
When $p\nmid N$, let $\pi_p'=\pi(\mu_p,\mu_p^{-1})$ for an unramified character $\mu_p$. Then
 $I(0,\mu_p^2)^{K_p}$ is one-dimensional and take its generator 
$\phi^{{\rm ur}}_p$ so that $\phi^{{\rm ur}}_p(K_p)=1$. When $p|N$, $\pi_p'={\rm St}_{GL_2}$ or its unramified quadratic twist. By Theorem 2.1 of \cite{MY} and the Frobenius reciprocity for the Jacquet modules with Proposition \ref{irr-uni}, 
we see that as a $\C$-vector space, 
$$A(|\cdot|_p)^{K^{\bold P}_p}\simeq \C,
$$
where $K^{\bold P}_p$ is the Iwahori parahoric subgroup of $\bold G(\Z_p)$ for the Siegel parabolic $\bold P$. 
Therefore there exists a generator $\phi^{{\rm st}}_p$ of $A(|\cdot|_p)^{K^{\bold P}_p}$. 
One can also  specify $\phi^{{\rm st}}_p$ explicitly by using  the Bruhat decomposition $\bold P(\Q_p)\bs \bold G(\Q_p)/K^{\bold P}_p$ consisting of four elements 
(see (1.6) of \cite{kim}) and intertwining operators. 
We do not pursue this and leave it to the interested readers. 
Put 
$$\phi:=\ds\bigotimes_{p|N} \phi^{{\rm st}}_p\otimes \bigotimes_{p\nmid N\infty}\phi^{{\rm ur}}_p.
$$
Then by Theorem \ref{main-thm1}, 
$\mathcal{F}(\cdot,\phi)$ is an automorphic function on $\bold G(\A_\Q)$ fixed by 
$K^{\bold P}_0(N):=\ds\prod_{p|N}K^{\bold P}_p
\times \prod_{p\nmid N\infty} K_p$ from the right. Now the claim follows from the fact that
$K^{\bold P}_0(N)\cap \bold G(\Q)=\Gamma^{\bold P}_0(N)$.  

\subsection{Proof of Theorem \ref{main-thm4}}\label{main-thm4-section}

Let $\Pi$ be as in Theorem \ref{main-thm4}. It is a cuspidal representation on $E_{7,3}^{\rm ad}$. We consider it as a cuspidal  representation of $GE_{7,3}$ with the trivial central character.

If $p\nmid N$, $\pi_p=\pi(\mu_{p},\mu_{p}^{-1})$, and $\Pi_p$ is the degenerate principal series ${\rm Ind}_{\bold{P}(\Q_p)}^{\bold{G}(\Q_p)}\: |\nu(\cdot)|^{2s_p}$, where $p^{s_p}=\mu_p(p)$. It is irreducible by Proposition \ref{irr-uni}.
Then
by inducing in stages, in the notation of \cite{KY}, $Ind_{R(\Bbb Q_p)}^{E_8(\Bbb Q_p)}\, \Pi_p\otimes exp(s\tilde\alpha, H_R(\ ))$ is a subrepresentation of
$$Ind_{B(\Bbb Q_p)}^{E_8(\Bbb Q_p)}\, exp(\chi, H_B(\ )),
$$
where $\chi=s(e_1-e_9)+s_p(-e_1+2e_2-e_9)+(8e_3+7e_4+6e_5+5e_6+4e_7+3e_8)$ in the notation of \cite[p. 249]{KY}. 
Then Langlands-Shahidi method gives rise to $L(s,\Pi_p,{\rm St})$ to be the right hand side of Theorem \ref{main-thm4}.

If $\pi_p={\rm St}_{GL_2}$, 
then $L(s,\pi_p)=L(s+\frac 12,1)=(1-p^{-s-\frac 12})^{-1}$, and $L(s,{\rm Sym}^3\pi_p)=L(s+\frac 32,1)$ by \cite{CM}. 
In this case, $\Pi_p$ is the subrepresentation of $Ind_{\bold P(\Bbb Q_p)}^{\bold G(\Bbb Q_p)}\, |\nu(\cdot)|$. 
Then by inducing in stages, in the notation of \cite{KY}, $Ind_{R(\Bbb Q_p)}^{E_8(\Bbb Q_p)}\, \Pi_p\otimes exp(s\tilde\alpha, H_R(\ ))$ is a subrepresentation of
$$Ind_{B(\Bbb Q_p)}^{E_8(\Bbb Q_p)}\, exp(\chi, H_B(\ )),
$$
where $\chi=s(e_1-e_9)+\frac 12(-e_1+2e_2-e_9)+(8e_3+7e_4+6e_5+5e_6+4e_7+3e_8)$ in the notation of \cite[p. 249]{KY}. 
Langlands-Shahidi method gives rise to the $\gamma$-factor, and defines the $L$-function as the reciprocal of the numerator of the 
$\gamma$-factor. Due to cancellation of the factors, we obtain $L(s,\Pi_p,{\rm St})$ to be the local factor of the right hand side in Theorem \ref{main-thm4}. 

\section{Degree 133 adjoint $L$-functions of the Ikeda type lift}\label{main-thm5-section}

We use the same notation as in \cite{kim1}:
We take simple roots, $\alpha_i=e_i-e_{i+1}$, $i=1,2,3,4,5,6$, $\alpha_7=e_5+e_6+e_7+e_8$. Here $(e_i,e_i)=\frac 78$, $(e_i,e_j)=-\frac 18$ for $1\leq i\ne j\leq 8$ and $\sum e_i=0$. The positive roots are $e_i-e_j$, $1\leq i<j\leq 7$, $-e_i+e_8$, $i=1,...,7$, and $e_i+e_j+e_k+e_8$. There are 63 of them. 
Let ${\rm Ad}$ be the degree 133 representation of $E_7(\Bbb C)$. It is the adjoint representation. Its highest weight is the one attached to the simple root $\alpha_6$, and its weights are 0 (with multiplicity 7), and roots of $E_7$.

Recall that if $p\nmid N$, the $p$-adic component of the Ikeda type lift $\Pi$ is given by
$$\Pi_p=2s_p(e_1+e_8)+8e_2+7e_3+6e_4+5e_5+4e_6+3e_7+11e_8.
$$
in the notation of \cite{kim1}.
Let $L(s,\pi_f)=\prod_p L(s,\pi_p)$, and for $p\nmid N$, $L(s,\pi_p)^{-1}=(1-\alpha_p p^{-s})(1-\alpha_p^{-1} p^{-s})$. 
Let $L(s,\pi_f,{\rm Sym}^2)$ be the symmetric square $L$-function. Then $L(s,\pi_f,{\rm Sym}^2)=\prod_p L(s,\pi_p, {\rm Sym}^2)$, 
and for $p\nmid N$, $L(s,\pi_p, {\rm Sym}^2)^{-1}=(1-\alpha_p^2 p^{-s})(1-p^{-s})(1-\alpha_p^{-2} p^{-s})$.  
For $p\nmid N$, the degree 133 local $L$-function is given by
\begin{eqnarray*}
&& L(s,\Pi_p,{\rm Ad})^{-1}=(1-p^{-s})^4 (1-p^{\pm 1-s})^4 (1-p^{\pm 2-s})^2 (1-p^{\pm 3-s})^4 (1-p^{\pm 4-s})^3 (1-p^{\pm 5-s})^3 \\
&& \phantom{xxxxx} (1-p^{\pm 6-s})^2 (1-p^{\pm 7-s})^2(1-p^{\pm 8-s})(1-p^{\pm 9-s})(1-p^{\pm 10-s})(1-p^{\pm 11-s}) \\
&& \phantom{xxxxx} \cdot L(s,\pi_p,{\rm Sym}^2)^{-3} \prod_{i=1}^4 L(s\pm i,\pi_p,{\rm Sym}^2)^{-2} \prod_{i=5}^8 L(s\pm i,\pi_p,{\rm Sym}^{2})^{-1}.
\end{eqnarray*}

If $p|N$, we follow Section \ref{main-thm4-section} to define the local factor. Namely, if $\pi_p={\rm St}_{GL_2}$, 
 $\Pi_p$ is the subrepresentation of $Ind_{\bold P(\Bbb Q_p)}^{\bold G(\Bbb Q_p)}\, |\nu(\cdot)|$. 
Then we first define $\gamma$-factors. Due to cancellation of the factors, we obtain $L(s,\Pi_p,{\rm Ad})$ as the reciprocal of the numerator of the 
$\gamma$-factors.
This proves Theorem \ref{main-thm5}.

\section{Ikeda type lift as CAP representation}

Let $f\in S_k(\Gamma_0(N))^{\rm new}$ for $N$ square free, and let $\pi_f=\otimes \pi_p$ be its associated cuspidal representation. 
If $p\nmid N$, let $\{\alpha_p,\alpha_p^{-1}\}$ be the Satake parameter of $\pi_p$. Let $\alpha_p=p^{s_p}$.

Let $F_f$ be the Ikeda lift of $f$ which is a cusp form on $G=Sp_{2n}$ (rank $2n$), and $\Pi_F=\otimes_p \Pi_p$ be the cuspidal representation associated to $F_f$. Then $\Pi_p=Ind_{P(\Bbb Q_p)}^{G(\Bbb Q_p)} |det|^{s_p}$, where $P$ is the Siegel parabolic subgroup such that $P=MN$, $M\simeq GL_{2n}$. The Ikeda lift $F_f$ is a CAP representation, namely, $\Pi_F$ is nearly equivalent to the quotient of the induced representation
$$Ind_{P_{2,...,2}}^{Sp_{2n}}\, \pi_f|det|^{n-\frac 12}\otimes\pi_f|det|^{n-\frac 32}\otimes\cdots\otimes \pi_f|det|^\frac 12,
$$
where $P_{2,...,2}$ is the standard parabolic subgroup of $Sp_{2n}$ with the Levi subgroup $GL_2\times\cdots\times GL_2$ ($n$ factors).

\smallskip

When $G=E_{7,3}^{\rm ad}$, we constructed the Ikeda type lift $F_f$ associated to $f$. Let $\Pi_F=\otimes_p \Pi_p$ be the cuspidal representation associated to 
$F_f$. Then for all finite prime $p$, $\bold G(\Bbb Q_p)=E_7^{\rm ad}(\Bbb Q_p)$ is the split group of type $E_7^{\rm ad}$, and
for $p\nmid N$,
$\Pi_p=Ind_{\bold P(\Bbb Q_p)}^{E_7^{\rm ad}(\Bbb Q_p)} |\nu(g)|^{2s_p}$, where $\bold P$ is the Siegel parabolic subgroup such that $\bold P=\bold M\bold N$, $\bold M\simeq GE_6$. Here $\Pi_F$ is not a CAP representation in the usual sense since there are not many $\Bbb Q$-parabolic subgroups of $E_{7,3}$. 

We show that $\Pi_F$ is a CAP representation in a more general sense, namely, there exists a parabolic group $\bold R=\bold M'\bold N'$ of the split $E_7^{\rm ad}$ and a cuspidal representation $\tau=\otimes_p \tau_p$ of $\bold M'(\Bbb A_\Q)$, and a parameter $\Lambda_0$ such that for almost all finite prime $p$,
$\Pi_p$ is a quotient of $Ind_{\bold R(\Bbb Q_p)}^{E_7^{\rm ad}(\Bbb Q_p)} \tau_p\otimes exp(\Lambda_0,H_Q(\, ))$.

Recall the Dynkin diagram of $E_{7}$:
\begin{align*}
{\text{o}}_{\beta_1}\text{------}{\text{o}}_{\beta_3}{%
\text{------}}&{\text{o}}_{\beta_4}{\text{------}}{\text{o}}_{\beta_5}{\text{------}}{\text{o}}_{\beta_6}{\text{------}}{%
\text{o}}_{\beta_7} \\
&\Big\vert \\
&{\text{o}}_{\beta_2}
\end{align*}

Let $\bold R$ be the parabolic subgroup attached to the set of simple roots $\{\beta_2,\beta_5,\beta_7\}$ so that $\bold R=\bold M'\bold N'$, and the derived group of $\bold M'$ is $\bold M_D'\simeq SL_2\times SL_2\times SL_2$.
Let $\tau$ be a cuspidal representation of $\bold M'(\Bbb A_\Q)$ induced by $\pi_f\otimes\pi_f\otimes\pi_f$, and $\Lambda_0=10\beta_1+11\beta_2+19\beta_3+26\beta_4+22\beta_5+15\beta_6+8\beta_7$.
Then when $\pi_p$ is unramified, 
$$Ind_{\bold R(\Bbb Q_p)}^{E_7^{\rm ad}(\Bbb Q_p)} \tau_p\otimes exp(\Lambda_0,H_R(\, ))=Ind_{\bold B(\Bbb Q_p)}^{E_7^{\rm ad}(\Bbb Q_p)}\, \chi,
$$
where $\chi=s_p\beta_2+s_p\beta_5+s_p\beta_7+\Lambda_0$. Then we can show that $\Pi_p$ is the quotient of $Ind_{\bold B(\Bbb Q_p)}^{E_7^{\rm ad}(\Bbb Q_p)}\, \chi$. Hence we have proved

\begin{theorem} $\Pi_F$ is a CAP representation in a general sense with respect to the parabolic subgroup $\bold R$.
\end{theorem}

\end{document}